\documentclass[11pt]{amsart}%
\usepackage[foot]{amsaddr}

\usepackage[a4paper, centering, vcentering]{geometry}
\usepackage{fullpage}

\usepackage{amsmath,amsfonts,amssymb,amsthm,mathrsfs}

\usepackage{hyperref,url}

\usepackage[centertableaux,smalltableaux]{ytableau}

\usepackage{mathtools} 
\mathtoolsset{showonlyrefs,showmanualtags} 

\usepackage{tikz}

\usepackage{xcolor}

\usepackage{xcolor}
\hypersetup{
    colorlinks,
    linkcolor={red!50!black},
    citecolor={blue!50!black},
    urlcolor={blue!80!black}
}

\newcommand{\blu}{\color{blue}}

\newtheorem{theorem}{Theorem}
\newtheorem{corollary}[theorem]{Corollary}
\newtheorem{proposition}[theorem]{Proposition}

\newtheorem{lemma}[theorem]{Lemma}
\newtheorem*{lemma*}{Lemma}
\theoremstyle{definition} 
\newtheorem{definition}[theorem]{Definition}
\newtheorem{example}{Example}
\theoremstyle{remark}
\newtheorem{remark}{Remark}

\newcommand{\bt}{\begin{theorem}}
\newcommand{\et}{\end{theorem}}
\newcommand{\bl}{\begin{lemma}}
\newcommand{\el}{\end{lemma}}
\newcommand{\bp}{\begin{proposition}}
\newcommand{\ep}{\end{proposition}}
\newcommand{\bc}{\begin{corollary}}
\newcommand{\ec}{\end{corollary}}
\newcommand{\bdeff}{\begin{definition}}
\newcommand{\edeff}{\end{definition}}
\newcommand{\brem}{\begin{remark}}
\newcommand{\erem}{\end{remark}}
\newcommand{\bex}{\begin{example}}
\newcommand{\eex}{\end{example}}
\newcommand{\bcen}{\begin{center}}
\newcommand{\ecen}{\end{center}}

\newcommand{\bi}{\begin{itemize}}

\newcommand{\ei}{\end{itemize}}
\newcommand{\bd}{\begin{description}}
\newcommand{\ed}{\end{description}}

\newcommand{\bqn}{\begin{eqnarray}}
\newcommand{\eqn}{\end{eqnarray}}

\newcommand{\la}{\langle}
\newcommand{\ra}{\rangle}

\newcommand{\g}{\gamma}

\newcommand{\eps}{\varepsilon}			
\newcommand{\lam}{\lambda}

\newcommand{\wt}[1]{\widetilde{#1}}
\newcommand{\mc}[1]{\mathcal{#1}}
\newcommand{\VecM}{\mathrm{Vec}(M)}				
\newcommand{\VecH}{\mathrm{Vec}_\distr(M)}		
\newcommand{\VecTM}{\mathrm{Vec}(T^*M)}			
\newcommand{\distr}{\mathscr{D}}				
\newcommand{\ve}{\mathcal{V}}					
\newcommand{\R}{\mathbb{R}}						
\newcommand{\lev}{\alpha}						
\newcommand{\RR}{\mathfrak{R}}					
\newcommand{\DD}{\mathscr{F}}					
\newcommand{\tanf}{\mathsf{T}}					
\newcommand{\y}{D}								
\renewcommand{\k}{k}							
\newcommand{\EXP}{\mc{E}}						
\newcommand{\barE}{\bar{E}}
\newcommand{\metr}{\la\cdot|\cdot\ra}			

\DeclareMathOperator{\Ric}{\mathfrak{Ric}}		
\DeclareMathOperator{\Sec}{\mathrm{Sec}}		
\DeclareMathOperator{\spn}{\mathrm{span}}		
\DeclareMathOperator{\trace}{\mathrm{tr}}		
\DeclareMathOperator{\diag}{\mathrm{diag}}		
\DeclareMathOperator{\rank}{\mathrm{rank}}		
\DeclareMathOperator{\diam}{\mathrm{diam}}		

\begin{document}

\title{Comparison theorems for conjugate points in sub-Riemannian geometry}

\author{D. Barilari$^1$}
\address{$^1$Universit\'e Paris Diderot - Paris 7, Institut de Mathematique de Jussieu, UMR CNRS 7586 - UFR de Math\'ematiques} \email{\href{mailto:davide.barilari@imj-prg.fr}{\nolinkurl{davide.barilari@imj-prg.fr}}}

\author{L. Rizzi$^2$}
\address{$^2$CNRS, CMAP \'Ecole Polytechnique and \'Equipe INRIA GECO Saclay \^Ile-de-France, Paris, France} \email{\href{mailto:luca.rizzi@cmap.polytechnique.fr}{\nolinkurl{luca.rizzi@cmap.polytechnique.fr}}}

\date{\today}



\begin{abstract}
We prove sectional and Ricci-type comparison theorems for the existence of conjugate points along sub-Riemannian geodesics. In order to do that, we regard sub-Riemannian structures as a special kind of variational problems. In this setting, we identify a class of models, namely linear quadratic optimal control systems, that play the role of the constant curvature spaces. As an application, we prove a version of sub-Riemannian Bonnet-Myers theorem and we obtain some new results on conjugate points for three dimensional left-invariant sub-Riemannian structures.
\end{abstract}

\maketitle

\setcounter{tocdepth}{2}
\tableofcontents

\section{Introduction}\label{s:intro}

Among the most celebrated results in Riemannian geometry, comparison theorems play a prominent role. These theorems allow to estimate properties of a manifold under investigation with the same property on the \emph{model spaces} which, in the classical setting, are the simply connected manifolds with constant sectional curvature (the sphere, the Euclidean plane and the hyperbolic plane). The properties that may be investigated with these techniques are countless and include, among the others, the number of conjugate points along a given geodesic, the topology of loop spaces, the behaviour of volume of sets under homotheties, 
Laplacian comparison theorems, estimates for solutions of PDEs on the manifold, etc.

In this paper we are concerned, in particular, with results of the following type. Until further notice, $M$ is a Riemannian manifold, endowed with the Levi-Civita connection, $\Sec(v,w)$ is the sectional curvature of the plane generated by $v,w \in T_x M$.

\begin{theorem}\label{t:model}
Let $\gamma(t)$ be a unit-speed geodesic on $M$. If for all $t\geq 0$ and for all $v \in T_{\gamma(t)} M$ orthogonal to $\dot{\gamma}(t)$ with unit norm $\Sec(\dot{\gamma}(t),v) \geq \k > 0$, then there exists $0<t_c\leq \pi/\sqrt{\k}$ such that $\gamma(t_c)$ is conjugate with $\gamma(0)$.
\end{theorem}
Notice that the quadratic form $\Sec(\dot{\gamma}(t),\cdot) : T_{\gamma(t)}M \to \R$, which we call \emph{directional curvature} (in the direction of $\dot{\gamma}$), computes the sectional curvature of the planes containing $\dot{\gamma}$. Theorem~\ref{t:model} compares the distance of the first conjugate point along $\gamma$ with the same property computed on the sphere with sectional curvature $\k>0$, provided that the directional curvature along the geodesic on the reference manifold is bounded from below by $\k$. Theorem~\ref{t:model} also contains all the basic ingredients of a comparison-type result:
\begin{itemize}
\item A micro-local condition, i.e. ``along the geodesic'', usually given in terms of curvature-type quantities, such as the sectional or Ricci curvature.
\item Models for comparison, that is spaces in which the property under investigation can be computed explicitly.
\end{itemize}
As it is well known, Theorem~\ref{t:model} can be improved by replacing the bound on the directional curvature with a bound on the average, or Ricci curvature. Moreover, Theorem~\ref{t:model} leads immediately to the celebrated Bonnet-Myers theorem (see \cite{Myers}).

\begin{theorem}\label{t:bm-riem-intro}
Let $M$ be a connected, complete Riemannian manifold, such that, for any unit-speed geodesic $\gamma(t)$, the Ricci curvature $\mathrm{Ric}^\nabla(\dot\gamma(t)) \geq n k$. Then, if $k > 0$, $M$ is compact, has diameter non greater than $\pi/\sqrt{\kappa}$ and its fundamental group is finite.
\end{theorem}

In Riemannian geometry, the importance of conjugate points rests on the fact that geodesics cease to be minimizing after the first one. This remains true for strongly normal sub-Riemannian geodesics. Moreover, conjugate points, both in Riemannian and sub-Riemannian geometry, are also intertwined with the analytic properties of the underlying structure, for example they affect the behaviour of the heat kernel (see \cite{bbcn,srneel} and references therein). 

The main results of this paper are comparison theorems on the existence of conjugate points, valid for any sub-Riemannian structure. 

We briefly introduce the concept of sub-Riemannian structure. A sub-Riemannian structure on a manifold $M$ can be defined as a distribution $\distr \subseteq TM$ of constant rank, with a scalar product that, unlike the Riemannian case, is defined only for vectors in $\distr$. Under mild assumptions on $\distr$ (the H\"ormander condition) any connected sub-Riemannian manifold is \emph{horizontally} path-connected, namely any two points are joined by a path whose tangent vector belongs to $\distr$. Thus, a rich theory paralleling the classical Riemannian one can be developed, giving a meaning to the concept of geodesic, as an horizontal curve that locally minimises the length. 

Still, since in general there is no canonical completion of the sub-Riemannian metric to a Riemannian one, there is no way to define a connection \`a la Levi-Civita and thus the familiar Riemannian curvature tensor. The classical theory of Jacobi fields and its connection with the curvature plays a central role in the proof of many Riemannian comparison results, and the generalisation to the sub-Riemannian setting is not straightforward. The Jacobi equation itself, being defined in terms of the covariant derivative, cannot be formalised in the classical sense when a connection is not available. 

In this paper we focus on results in the spirit of Theorem~\ref{t:model} even tough there are no evident obstructions to the application of the same techniques, relying on the Riccati equations for sub-Riemannian geodesics, to other comparison results. We anticipate that the comparisons models will be linear quadratic optimal control problems (LQ problems in the following), i.e. minimization problems quite similar to the Riemannian one, where the length is replaced by a functional defined by a quadratic Lagrangian. More precisely we are interested in \emph{admissible trajectories} of a \emph{linear} control system in $\R^n$, namely  curves $x:[0,t]\to \mathbb{R}^n$ for which there exists a control $u \in L^2([0,t],\mathbb{R}^k)$ such that
\begin{equation}
\dot{x} = Ax +Bu, \qquad x(0) = x_0, \qquad x(t) = x_1,\qquad x_0,x_1,t \text{ fixed},
\end{equation}
that minimize a \emph{quadratic} functional $\phi_{t}: L^2([0,t],\mathbb{R}^k) \to \mathbb{R}$ of the form
\begin{equation}
\phi_{t}(u) = \frac{1}{2}\int_{0}^{t} \left(u^* u - x^*Qx \right)dt.
\end{equation}
Here $A,B,Q$ are constant matrices of the appropriate dimension. 
 The symmetric matrix $Q$ is usually referred to as the \emph{potential}. Notice that it makes sense to speak about \emph{conjugate time} of a LQ problem: it is the time $t_c>0$ at which extremal trajectories lose local optimality, as in (sub)-Riemannian geometry. Moreover, $t_c$ does not depend on the data $x_0,x_1$, but it is an intrinsic feature of the problem. These kind of structures are well known in the field of optimal control theory, but to our best knowledge this is the first time they are employed as model spaces for comparison results.

With any ample, equiregular sub-Riemannian geodesic $\gamma(t)$ (see Definition~\ref{d:ampleq}), we associate: its \emph{Young diagram} $\y$, a scalar product $\langle\cdot|\cdot\rangle_{\gamma(t)} : T_{\gamma(t)} M \times T_{\gamma(t)} M\mapsto \R$ extending the sub-Riemannian one and a quadratic form $\RR_{\gamma(t)} : T_{\gamma(t)} M \mapsto \R$ (the sub-Riemannian directional curvature), all depending on the geodesic $\gamma(t)$. We stress that, for a Riemannian manifold, any non-trivial geodesic has the same Young diagram, composed by a single column with $n=\dim M$ boxes, the scalar product $\langle\cdot|\cdot\rangle_{\gamma(t)}$ coincides with the Riemannian one, and $\RR_{\gamma(t)}(v) = \Sec(v,\dot{\gamma}(t))$. 

In this introduction, when we associate with a geodesic $\g(t)$ its Young diagram $\y$, we implicitly assume that $\g(t)$ is ample and equiregular. Notice that these assumptions are true for the generic geodesic, as we discuss more precisely in Sec.~\ref{s:gfyd}.

In the spirit of Theorem~\ref{t:model}, assume that the sub-Riemannian directional curvature is bounded from below by a quadratic form $Q$. Then, we associate a model LQ problem (i.e. matrices $A$ and $B$, depending on $\gamma$) which, roughly speaking, represents the linearisation of the sub-Riemannian structure along $\gamma$ itself,  with potential $Q$. We dub this \emph{model space} $\mathrm{LQ}(\y; Q)$, where $\y$ is the Young diagram of $\gamma$, and $Q$ represents the bound on the sub-Riemannian directional curvature. The first of our results can be stated as follows (see Theorem~\ref{t:comparison1}).
\begin{theorem}[sub-Riemannian comparison]\label{t:model2}
Let $\gamma(t)$ be a sub-Riemannian geodesic, with Young diagram $\y$, such that $\RR_{\gamma(t)} \geq Q_+$ for all $t\geq 0$. Then the first conjugate point along $\gamma(t)$ occurs at a time $t$ not greater than the first conjugate time of the model $\mathrm{LQ}(\y;Q_+)$. Similarly, if $\RR_{\gamma(t)} \leq Q_-$ for all $t\geq 0$, the first conjugate point along $\gamma(t)$ occurs at a time $t$ not smaller than the first conjugate time of $\mathrm{LQ}(\y;Q_-)$.
\end{theorem}

In the Riemannian case, any non-trivial geodesic $\gamma$ has the same (trivial) Young diagram, and this leads to a simple LQ model with $A=0$ and $B=\mathbb{I}$ the identity matrix. Moreover, $\langle\cdot|\cdot\rangle_\gamma$ is the Riemannian scalar product and $\RR_\gamma = \Sec(\dot{\gamma},\cdot)$. Then, if Theorem~\ref{t:model2} holds with $Q_+ = \k \mathbb{I}$, the first conjugate point along the Riemannian geodesic, with directional curvature bounded by $\k$ occurs at a time $t$ not greater than the first conjugate time of the LQ model
\begin{equation}
\dot{x} = u, \qquad \phi_{t}(u) = \frac{1}{2}\int_{0}^{t}\left( |u|^2 - \k |x|^2\right)dt.
\end{equation}
It is well known that, when $\k >0$, this problem represents a simple $n$-dimensional harmonic oscillator, whose extremal trajectories lose optimality at time $t = \pi / \sqrt{\k}$. Thus we recover Theorem~\ref{t:model}. On the other hand, in the sub-Riemannian setting, due to the intrinsic anisotropy of the structure different geodesics have different Young diagrams, resulting in a rich class of LQ models, with non-trivial drift terms. The directional sub-Riemannian curvature represents the potential ``experienced'' in a neighbourhood of the geodesic.

We stress that the generic $\mathrm{LQ}(\y;Q)$ model may have infinite conjugate time. However, there exist necessary and sufficient conditions for its finiteness, that are the sub-Riemannian counterpart of the ``Riemannian'' condition $\k>0$ of Theorem~\ref{t:model}. Thus Theorem~\ref{t:model2} can be employed to prove both existence or non-existence of conjugate points along a given geodesic.


As Theorem~\ref{t:model} can be improved by considering a bound on the Ricci curvature in the direction of the geodesic, instead of the whole sectional curvature, also Theorem~\ref{t:model2} can be improved in the same spirit. In the sub-Riemannian case, however, the process of ``taking the trace'' is more delicate. Due to the anisotropy of the structure, it only makes sense to take \emph{partial} traces, leading to a number of Ricci curvatures (each one obtained as a partial trace on an invariant subspace of $T_{\gamma(t)} M$, determined by the Young diagram $\y$). In particular, for each \emph{level} $\alpha$ of the Young diagram (namely the collection of all the rows with the same length equal to, say, $\ell$) we have $\ell$ Ricci curvatures $\Ric_{\gamma(t)}^{\alpha_i}$, for $i=1,\ldots,\ell$. The \emph{size} of a level is the number $r$ of boxes in each of its columns $\lev_1,\ldots,\lev_\ell$.
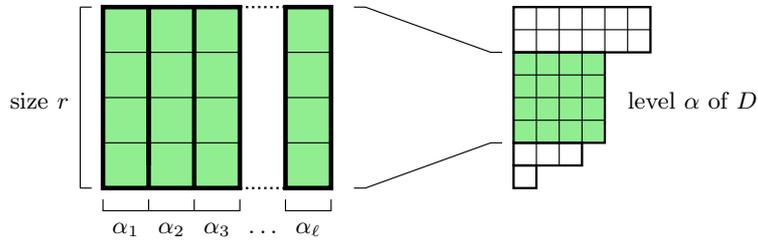
\begin{center}
\begin{tikzpicture}[x=0.30mm, y=0.30mm, inner xsep=0pt, inner ysep=0pt, outer xsep=0pt, outer ysep=0pt]
\path[line width=0mm] (33.62,67.84) rectangle +(226.93,102.16);
\definecolor{L}{rgb}{0,0,0}
\definecolor{F}{rgb}{0.565,0.933,0.565}
\path[line width=0.60mm, draw=L, fill=F] (80.00,90.00) rectangle +(60.00,80.00);
\path[line width=0.60mm, draw=L] (100.00,170.00) -- (100.00,100.00);
\path[line width=0.60mm, draw=L] (100.00,170.00) -- (100.00,90.00);
\path[line width=0.60mm, draw=L] (120.00,170.00) -- (120.00,90.00);
\path[line width=0.30mm, draw=L, dash pattern=on 0.30mm off 0.50mm] (140.00,170.00) -- (160.00,170.00);
\path[line width=0.30mm, draw=L, dash pattern=on 0.30mm off 0.50mm] (140.00,90.00) -- (160.00,90.00);
\path[line width=0.60mm, draw=L, fill=F] (160.00,90.00) rectangle +(20.00,80.00);
\path[line width=0.15mm, draw=L] (80.00,150.00) -- (140.00,150.00);
\path[line width=0.15mm, draw=L] (80.00,130.00) -- (140.00,130.00);
\path[line width=0.15mm, draw=L] (80.00,110.00) -- (140.00,110.00);
\path[line width=0.15mm, draw=L] (160.00,150.00) -- (180.00,150.00);
\path[line width=0.15mm, draw=L] (160.00,130.00) -- (180.00,130.00);
\path[line width=0.15mm, draw=L] (160.00,110.00) -- (180.00,110.00);
\draw(90.00,70.00) node[anchor=base]{\fontsize{8.54}{10.24}\selectfont $\lev_1$};
\draw(110.00,70.00) node[anchor=base]{\fontsize{8.54}{10.24}\selectfont $\lev_2$};
\draw(130.00,70.00) node[anchor=base]{\fontsize{8.54}{10.24}\selectfont $\lev_3$};
\draw(170.00,70.00) node[anchor=base]{\fontsize{8.54}{10.24}\selectfont $\lev_\ell$};
\draw(150.00,70.00) node[anchor=base]{\fontsize{8.54}{10.24}\selectfont $\ldots$};
\path[line width=0.15mm, draw=L] (80.00,85.00) -- (80.00,80.00) -- (140.00,80.00) -- (140.00,85.00);
\path[line width=0.15mm, draw=L] (160.00,85.00) -- (160.00,80.00) -- (180.00,80.00) -- (180.00,85.00);
\path[line width=0.15mm, draw=L] (100.00,85.00) -- (100.00,80.00);
\path[line width=0.15mm, draw=L] (120.00,85.00) -- (120.00,80.00);
\path[line width=0.15mm, draw=L] (75.00,170.00) -- (70.00,170.00);
\path[line width=0.15mm, draw=L] (75.00,90.00) -- (70.00,90.00);
\path[line width=0.15mm, draw=L] (70.00,170.00) -- (70.00,90.00);
\path[line width=0.15mm, draw=L] (240.00,145.00) rectangle +(0.00,0.00);
\path[line width=0.30mm, draw=L, fill=F] (260.00,110.00) rectangle +(40.00,40.00);
\path[line width=0.30mm, draw=L] (260.00,90.00) rectangle +(10.00,10.00);
\path[line width=0.30mm, draw=L] (260.00,100.00) rectangle +(30.00,10.00);
\path[line width=0.30mm, draw=L] (260.00,150.00) rectangle +(60.00,20.00);
\path[line width=0.15mm, draw=L] (260.00,160.00) -- (320.00,160.00);
\path[line width=0.15mm, draw=L] (260.00,140.00) -- (300.00,140.00);
\path[line width=0.15mm, draw=L] (260.00,130.00) -- (300.00,130.00);
\path[line width=0.15mm, draw=L] (260.00,120.00) -- (300.00,120.00);
\path[line width=0.15mm, draw=L] (270.00,170.00) -- (270.00,100.00);
\path[line width=0.15mm, draw=L] (280.00,170.00) -- (280.00,100.00);
\path[line width=0.15mm, draw=L] (290.00,170.00) -- (290.00,110.00);
\path[line width=0.15mm, draw=L] (300.00,150.00) -- (300.00,170.00);
\path[line width=0.15mm, draw=L] (310.00,150.00) -- (310.00,170.00);
\path[line width=0.15mm, draw=L] (255.00,150.00) -- (250.00,150.00) -- (195.00,170.00) -- (190.00,170.00);
\path[line width=0.15mm, draw=L] (190.00,90.00) -- (195.00,90.00) -- (250.00,110.00) -- (255.00,110.00);
\draw(65.00,125.00) node[anchor=base east]{\fontsize{8.54}{10.24}\selectfont size $r$};
\draw(310.00,125.00) node[anchor=base west]{\fontsize{8.54}{10.24}\selectfont level $\alpha$ of $\y$};
\end{tikzpicture}%

\end{center}
The partial tracing process leads to our main result (see Theorem~\ref{t:comparisonaverage}). 

\begin{theorem}[sub-Riemannian average comparison]\label{t:model3}
Let $\gamma(t)$ be a sub-Riemannian geodesic with Young diagram $\y$. Consider a fixed level $\lev$ of $\y$, with length $\ell$  and size $r$. Then, if
\begin{equation}
\frac{1}{r}\Ric_{\gamma(t)}^{\lev_i} \geq \k_i, \qquad \forall i=1,\ldots,\ell, \qquad \forall t \geq 0,
\end{equation}
the first conjugate time $t_c(\g)$ along the geodesic satisfies $t_c(\g) \leq t_c(\k_1,\ldots,\k_\ell)$.
\end{theorem}
In Theorem~\ref{t:model3}, $t_c(\k_1,\ldots,\k_\ell)$ is the first conjugate time of the LQ model associated with a Young diagram 	with a single row, of length $\ell$, and a diagonal potential $Q = \diag\{\k_1,\ldots,\k_\ell\}$. 

The hypotheses in Theorem~\ref{t:model3} are no longer bounds on a quadratic form as in Theorem~\ref{t:model2}, but a finite number of \emph{scalar} bounds. Observe that we have one comparison theorem for each level of the Young diagram of the given geodesic. In the Riemannian case, as we discussed earlier, $\y$ has only one level, of length $\ell =1$, of size $r=\dim M$. In this case there is single Ricci curvature, namely $\Ric_{\gamma(t)}^{\alpha_1} = \mathrm{Ric}^{\nabla}(\dot{\gamma}(t))$ and, if $\k_1> 0$ in Theorem~\ref{t:model3}, $t_c(\k_1) = \pi/\sqrt{\k_1} < +\infty$. We stress that, in order to have $t_c(\k_1,\ldots,\k_\ell) < +\infty$, the Riemannian condition $\mathrm{Ric}^{\nabla}(\dot{\gamma}) \geq \k_1 > 0$ must be replaced by more complicated inequalities on the bounds $\k_1,\ldots,\k_\ell$ on the sub-Riemannian Ricci curvatures. In particular, we allow also for some \emph{negative values} of such constants.

As an application of Theorem~\ref{t:model3}, we prove a sub-Riemannian version of the classical Bonnet-Myers theorem (see Theorem~\ref{t:bonnetmyers}).

\begin{theorem}[sub-Riemannian Bonnet-Myers]\label{t:model4}
Let $M$ be a connected, complete sub-Riemannian manifold, such that the generic geodesic has the same Young diagram $\y$. Assume that there exists a level $\alpha$ of length $\ell$ and size $r$ and constants $\k_1,\ldots,\k_\ell$ such that, for any length parametrized geodesic $\gamma(t)$
\begin{equation}
\frac{1}{r}\Ric^{\alpha_i}_{\gamma(t)} \geq \k_i, \qquad \forall i =1,\ldots,\ell, \qquad \forall t \geq 0.
\end{equation}
Then, if the polynomial
\begin{equation}
P_{\k_1,\ldots,\k_\ell}(x) := x^{2\ell} - \sum_{i=0}^{\ell -1} (-1)^{\ell-i}\k_{\ell -i} x^{2i}
\end{equation}
has at least one simple purely imaginary root, the manifold is compact, has diameter not greater than $t_c(\k_1,\ldots,\k_\ell)<+\infty$. Moreover, its fundamental group is finite.
\end{theorem}
In the Riemannian setting $\ell = 1$, $r = \dim M$ and the condition on the roots of $P_{\k_1}(x) = x^2 + \k_1$ is equivalent to $\k_1>0$. Then we recover the classical Bonnet-Myers theorem (see Theorem~\ref{t:bm-riem-intro}).

Finally we apply our techniques to obtain information about the conjugate time of geodesics on 3D \emph{unimodular} Lie groups. Left-invariant structures on 3D Lie groups are the basic examples of sub-Riemannian manifolds and the study of such structures is the starting point to understand the general properties of sub-Riemannian geometry. 

A complete classification of such structures, up to local sub-Riemannian isometries, is given in \cite[Thm. 1]{miosr3d}, in terms of the two basic geometric invariants $\chi\geq 0 $ and $\kappa$, that are constant for left-invariant structures. In particular, for each choice of the pair $(\chi,\kappa)$, there exists a unique unimodular group in this classification.
Even if left-invariant structures possess the symmetries inherited by the group structure, the sub-Riemannian geodesics and their conjugate loci have been studied only in some particular cases where explicit computations are possible. 

The conjugate locus of left-invariant structures has been completely determined for the cases corresponding to $\chi=0$, that are the Heisenberg group \cite{gersh} and the semisimple Lie groups $\mathrm{SU}(2),\mathrm{SL}(2)$ where the metric is defined by the Killing form \cite{boscainrossi}.
On the other hand, when $\chi>0$, only few cases have been considered up to now. In particular, to our best knowledge, only the sub-Riemannian structure on the group of motions of the Euclidean (resp. pseudo-Euclidean) plane $\mathrm{SE}(2)$ (resp. $\mathrm{SH}(2)$), where $\chi=\kappa>0$ (resp. $\chi = -\kappa >0$), has been considered \cite{sachkovmois,sachkov,sachkovpseudo}. 

As an application of our results, we give an explicit sufficient condition for a geodesic $\gamma$ on a unimodular Lie group to have a finite conjugate time, together with an estimate of it. The condition is expressed in terms of a lower bound on a constant of the motion $E(\gamma)$ associated with the given geodesic (see Theorem~\ref{t:Egrande}).

\begin{theorem}[Conjugate points for 3D structures]
Let $M$ be a 3D unimodular Lie group endowed with a contact left-invariant sub-Riemannian structure with invariants $\chi>0$ and $\kappa\in\R$. Then there exists $\barE=\barE(\chi,\kappa)$ such that every length parametrized geodesic $\g$ with $E(\gamma)\geq \barE$ has a finite conjugate time.  
\end{theorem}

The cases corresponding to $\chi=0$ are $\mathbb{H}$, $\mathrm{SU}(2)$ and $\mathrm{SL}(2)$, where $\kappa=0,1,-1$, respectively. For these structures we recover the exact estimates for the first conjugate time of a length parametrized geodesic (see Section \ref{s:chi0}).

\subsection{Related literature}

The curvature employed in this paper has been introduced for the first time by Agrachev and Gamkrelidze in \cite{agrafeedback}, Agrachev and Zelenko in \cite{geometryjacobi1} and successively extended by Zelenko and Li in \cite{lizel}, where also the Young diagram is introduced for the first time in relation with the extremals of a variational problem. This paper is not the first one to investigate comparison-type results on sub-Rieman\-nian manifolds, but has been inspired by many recent works in this direction that we briefly review.

In \cite{AAPL} Agrachev and Lee investigate a generalisation of the measure contraction property (MCP) to 3D sub-Riemannian manifolds. The generalised MCP of Agrachev and Lee is expressed in terms of solutions of a particular 2D matrix Riccati equation for sub-Riemannian extremals, and this is one of the technical points that mostly inspired the present paper.

In \cite{mcpcontact} Lee, Li and Zelenko pursue further progresses for sub-Riemannian Sasakian contact structures, which posses transversal symmetries. In this case, it is possible to exploit the Riemannian structure induced on the quotient space to write the curvature operator, and the authors recover sufficient condition for the contact manifold to satisfy the generalised MCP defined in \cite{AAPL}. Moreover, the authors perform the first step in the decoupling of the matrix Riccati equation for different levels of the Young diagram (see the splitting part of the proof of Theorem~\ref{t:comparisonaverage} for more details).

The MCP for higher dimensional sub-Riemannian structures has also been investigated in \cite{RiffordCarnot} for Carnot groups.

We also mention that, in \cite{lizel2}, Li and Zelenko prove comparison results for the number of conjugate points of curves in a Lagrange Grassmanian associated with sub-Riemannian structures with symmetries. In particular, \cite[Cor. 4]{lizel2} is equivalent to Theorem~\ref{t:model2}, but obtained with differential topology techniques and with a different language. However, to our best knowledge, it is not clear how to obtain an averaged version of such comparison results with these techniques, and this is yet another motivation that led to Theorem~\ref{t:model3}. 

In \cite{garofalob}, Baudoin and Garofalo prove, with heat-semigroup techniques, a sub-Riemannian version of the Bonnet-Myers theorem for sub-Riemannian manifolds with transverse symmetries that satisfy an appropriate generalisation of the Curvature Dimension (CD) inequalities introduced in the same paper. In \cite{baudoincontact}, Baudoin and Wang generalise the previous results to contact sub-Riemannian manifolds, removing the symmetries assumption. See also \cite{baudoinmunivegarofalo,garobonneinequalities} for other comparison results following from the generalised CD condition.

Even though in this paper we discuss only sub-Riemannian structures, these techniques can be applied to the extremals of any affine optimal control problem, a general framework including (sub)-Riemannian, (sub)-Finsler manifolds, as discussed in \cite{curvature}. 
For example, in \cite{agrafeedback}, the authors prove a comparison theorem for conjugate points along extremals associated with \emph{regular} Hamiltonian systems, such as those corresponding to Riemannian and Finsler geodesics. Finally, concerning comparison theorems for Finsler structures one can see, for example, \cite{ohtanocit,ohtasturm,wuxin}.

\subsection{Structure of the paper}
The plan of the paper is as follows. In Sec.~\ref{s:prel} we provide the basic definitions of sub-Riemannian geometry, and in particular the growth vector and the Young diagram of a sub-Riemannian geodesic. In Sec.~\ref{s:Jac} we revisit the theory of Jacobi fields. In Sec.~\ref{s:microlocal} we introduce the main technical tool, that is the generalised matrix Riccati equation, and the appropriate comparison models. Then, in Sec.~\ref{s:average} we provide the ``average'' version of our comparison theorems, transitioning from sectional-curvature type results to Ricci-curvature type ones. In Sec.~\ref{s:bm}, as an application, we prove a sub-Riemannian Bonnet-Myers theorem. Finally, in Sec.~\ref{s:3d}, we apply our theorems to obtain some new results on conjugate points for 3D left-invariant sub-Riemannian structures.
\section{Preliminaries}\label{s:prel}

Let us recall some basic facts in sub-Riemannian geometry. We refer to \cite{nostrolibro} for further details. 

Let $M$ be a smooth, connected manifold of dimension $n \geq 3$. A sub-Riemannian structure on $M$ is a pair $(\distr,\metr)$ where $\distr$ is a smooth vector distribution of constant rank $k\leq n$ satisfying the \emph{H\"ormander condition} (i.e. $\mathrm{Lie}_x\distr = T_x M$, $\forall x \in M$) and $\metr$ is a smooth Riemannian metric on $\distr$. A Lipschitz continuous curve $\g:[0,T]\to M$ is \emph{horizontal} (or \emph{admissible}) if $\dot\g(t)\in\distr_{\g(t)}$ for a.e. $t \in [0,T]$.
Given a horizontal curve $\g:[0,T]\to M$, the \emph{length of $\g$} is
\begin{equation}
\ell(\g)=\int_0^T \|\dot{\g}(t)\|dt,
\end{equation}
where $\|\cdot\|$ denotes the norm induced by $\metr$. The \emph{sub-Riemannian distance} is the function
\begin{equation}
d(x,y):=\inf \{\ell(\g)\mid \g(0)=x,\g(T)=y, \g\, \mathrm{horizontal}\}.
\end{equation}
The connectedness of $M$ and the H\"ormander condition guarantee the finiteness and the continuity of $d :M\times M \to \R$ with respect to the topology of $M$ (Rashevsky-Chow theorem). The space of vector fields on $M$ (smooth sections of $TM$) is denoted by $\VecM$. Analogously, the space of horizontal vector fields on $M$ (smooth sections of $\distr$) is denoted by $\VecH$.

\begin{example}
A sub-Riemannian manifold of odd dimension is \emph{contact} if $\distr= \ker \omega$, where $\omega$ is a one-form and $d\omega|_{\distr} $ is non degenerate. The \emph{Reeb vector field} $X_0 \in \VecM$ is the unique vector field such that $d\omega(X_0,\cdot) = 0$ and $\omega(X_0) = 1$.
\end{example}
\begin{example}
Let $M$ be a Lie group, and $L_x : M \to M$ be the left translation by $x \in M$. A sub-Riemannian structure $(\distr,\metr)$ is \emph{left-invariant} if $d_y L_x : \distr_y \to \distr_{L_x y}$ and is an isometry w.r.t. $\metr$ for all $x,y \in M$. Any Lie group admits left invariant structures obtained by choosing a scalar product on its Lie algebra and transporting it on the whole $M$ by left translation.
\end{example}

Locally, the pair $(\distr,\metr)$ can be given by assigning a set of $k$ smooth vector fields that span $\distr$, orthonormal for $\metr$. In this case, the set $\{X_1,\ldots,X_k\}$ is called a \emph{local orthonormal frame} for the sub-Riemannian structure. Finally, we can write the system in ``control form'', namely for any horizontal curve $\gamma:[0,T] \to M$ there is a \emph{control} $u \in L^\infty([0,T],\R^k)$ such that
\begin{equation}\label{eq:lnonce0}
\dot \gamma(t)=\sum_{i=1}^k u_i(t) X_i|_{\gamma(t)}, \qquad \text{a.e. } t \in [0,T].
\end{equation}

\subsection{Minimizers and geodesics}
A sub-Riemannian \emph{geodesic} is an admissible curve $\g:[0,T]\to M$ such that $\|\dot\g(t)\|$ is constant and for every sufficiently small interval $[t_1,t_2]\subseteq [0,T]$, the restriction $\g|_{[t_1,t_2]}$ minimizes the length between its endpoints. The length of a geodesic is invariant by reparametrization of the latter. Geodesics for which $\|\dot\g(t)\| = 1$ are called \emph{length parametrized} (or of \emph{unit speed}). A sub-Riemannian manifold is said to be \emph{complete} if $(M,d)$ is complete as a metric space.

With any sub-Riemannian structure we associate the Hamiltonian function $H \in C^\infty(T^*M)$
\begin{equation}
H(\lambda) = \frac{1}{2}\sum_{i=1}^k\la \lambda, X_i\ra^2, \qquad \forall \lambda \in T^*M,
\end{equation}
in terms of any local frame $X_1,\ldots,X_k$, where $\la\lambda,\cdot\ra$ denotes the action of the covector $\lambda$ on vectors.  Let $\sigma$ be the canonical symplectic form on $T^*M$. With the symbol $\vec{a}$ we denote the Hamiltonian vector field on $T^*M$ associated with a function $a \in C^\infty(T^*M)$. Indeed $\vec{a}$ is defined by the formula $da = \sigma(\cdot,\vec{a})$. For $i=1,\ldots,k$ let $h_i \in C^\infty(T^*M)$ be the linear-on-fibers functions defined by $h_i(\lambda):= \la \lambda,X_i\ra$. Notice that
\begin{equation}
H = \frac{1}{2}\sum_{i=1}^k h_i^2, \qquad \vec{H} = \sum_{i=1}^k h_i \vec{h}_i.
\end{equation}

Trajectories minimizing the distance between two points are solutions of first-order necessary conditions for optimality, which in the case of sub-Riemannian geometry are given by a weak version of the Pontryagin Maximum Principle (\cite{pontrybook}, see also \cite{nostrolibro} for an elementary proof). We denote by $\pi:T^*M \to M$ the standard bundle projection.

\begin{theorem}\label{t:pmpw}
Let $\gamma:[0,T] \to M$ be a sub-Riemannian geodesic associated with a non-zero control $u \in L^\infty([0,T],\R^k)$. Then there exists a Lipschitz curve $\lambda: [0,T] \to
T^*M$, such that $\pi \circ \lambda = \gamma$ and only one of the following conditions holds for a.e. $t \in [0,T]$:
\begin{itemize}
\item[(i)] $\dot\lambda(t) = \vec{H}|_{\lambda(t)}$ and $h_i(\lambda(t)) = u_i(t)$,
\item[(ii)] $\dot\lambda(t) = \displaystyle \sum_{i=1}^k u_i(t) \vec{h}_i|_{\lambda(t)}$, $\lambda(t) \neq 0$ and $h_i(\lambda(t)) = 0$.
\end{itemize}
\end{theorem}
If $\lambda:[0,T] \to M$ is a solution of (i) (resp. (ii)) it is called a \emph{normal} (resp. \emph{abnormal}) \emph{extremal}). It is well known that if $\lambda(t)$ is a normal extremal, then its projection $\gamma(t):=\pi(\lambda(t))$ is a smooth geodesic. This does not hold in general for abnormal extremals. On the other hand, a geodesic can be at the same time normal and abnormal, namely it admits distinct extremals, satisfying (i) and (ii). In the Riemannian setting there are no abnormal extremals.

\begin{definition}
A geodesic $\gamma:[0,T]\to M$ is \emph{strictly normal} if it is not abnormal. It is called \emph{strongly normal} if for every $t \in (0,T]$, the segment $\gamma|_{[0,t]}$ is not abnormal.
\end{definition}

Notice that extremals satisfying (i) are simply integral lines of the Hamiltonian field $\vec{H}$. Thus, let $\lambda(t)=e^{t\vec{H}}(\lambda_0)$ denote the integral line of $\vec{H}$, with initial condition $\lambda(0) = \lambda_0$. The sub-Riemannian \emph{exponential map} starting from $x_{0}$ is
\begin{equation}\label{eq:expmap}
\EXP_{x_{0}}: T_{x_0}^*M \to M, \qquad \EXP_{x_{0}}(\lambda_{0}):= \pi(e^{\vec{H}}(\lambda_{0})).
\end{equation}
Unit speed normal geodesics correspond to initial covectors $\lambda_{0}\in T^*M$ such that $H(\lambda) = 1/2$. 

\subsection{Geodesic flag and Young diagram}\label{s:gfyd}

In this section we introduce a set of invariants of a sub-Riemannian geodesic, namely the geodesic flag, and a useful graphical representation of the latter: the Young diagram. The concept of Young diagram in this setting appeared for the first time in \cite{lizel}, as a fundamental invariant for curves in the Lagrange Grassmanian. 
The proof that the original definition in \cite{lizel} is equivalent to the forthcoming one can be found in \cite[Sec. 6]{curvature}, in the general setting of affine control systems.

Let $\gamma(t)$ be a normal sub-Riemannian geodesic. By definition $\dot{\gamma}(t) \in \distr_{\g(t)}$ for all times. Consider a smooth horizontal extension of the tangent vector, namely an horizontal vector field $\tanf \in \VecH$ such that $\tanf|_{\gamma(t)} = \dot{\gamma}(t)$.
\begin{definition}\label{d:flag}
The \emph{flag of the geodesic} $\gamma(t)$ is the sequence of subspaces
\begin{equation}
\DD_\gamma^i(t) :=  \spn\{\mc{L}_\tanf^j (X)|_{\gamma(t)} \mid  X \in \VecH,\, j \leq i-1\} \subseteq T_{\gamma(t)} M, \qquad i \geq 1,
\end{equation}
where $\mc{L}_{\tanf}$ denotes the Lie derivative in the direction of $\tanf$.
\end{definition}
By definition, this is a filtration of $T_{\gamma(t)}M$, i.e. $\DD_\gamma^i(t) \subseteq \DD_\gamma^{i+1}(t)$, for all $i \geq 1$. Moreover, $\DD_\gamma^1(t) = \distr_{\gamma(t)}$. Definition~\ref{d:flag} is well posed, namely does not depend on the choice of the horizontal extension $\tanf$ (see \cite[Sec. 3.4]{curvature}).

For each time $t$, the flag of the geodesic contains information about how new directions can be obtained by taking the Lie derivative in the direction of the geodesic itself. In this sense it carries information about the germ of the distribution along the given trajectory, and is the microlocal analogue of the flag of the distribution.

\begin{definition}\label{d:growth}
The \emph{growth vector} of the geodesic $\gamma(t)$ is the sequence of integer numbers 
\begin{equation}
\mathcal{G}_\gamma(t) := \{\dim \DD_\gamma^1(t),\dim \DD_\gamma^2(t),\ldots\}.
\end{equation}
\end{definition}
Notice that, by definition, $\dim \DD_\gamma^1(t) = \dim \distr_{\gamma(t)} =k$.
\begin{definition}\label{d:ampleq}
Let $\gamma(t)$ be a normal sub-Riemannian geodesic, with growth vector $\mathcal{G}_\gamma(t)$. We say that the geodesic is:
\begin{itemize}
\item \emph{equiregular} if $\dim \DD_\gamma^i(t)$ does not depend on $t$ for all $i \geq 1$,
\item \emph{ample} if for all $t$ there exists $m \geq 1$ such that $\dim \DD_\gamma^{m}(t) = \dim T_{\gamma(t)}M$.
\end{itemize}
\end{definition}
We stress that equiregular (resp. ample) geodesics are the microlocal counterpart of equiregular (resp. bracket-generating) distributions. Let $d_i:= \dim \DD_\gamma^i - \dim \DD_\gamma^{i-1}$, for $i\geq 1$ be the increment of dimension of the flag of the geodesic at each step (with the convention $k_0:=0$).
\begin{lemma}\label{l:decreasing}
For an equiregular, ample geodesic, $d_1 \geq d_2 \geq \ldots \geq d_m$.
\end{lemma} 
\begin{proof}
By the equiregularity assumption, the Lie derivative $\mc{L}_\tanf$ defines surjective linear maps
\begin{equation}
\mc{L}_\tanf: \DD_\gamma^i(t) / \DD_\gamma^{i-1}(t) \to  \DD_\gamma^{i+1}(t) / \DD_\gamma^{i}(t), \qquad \forall t, \quad i \geq 1,
\end{equation}
where we set $\DD_\gamma^0(t) = \{0\}$ (see also \cite[Sec. 3.4]{curvature}). The quotients $\DD_\gamma^i / \DD_\gamma^{i-1}$ have constant dimension $d_i:= \dim \DD_\gamma^i - \dim \DD_\gamma^{i-1}$. Therefore the sequence $d_1\geq d_2\geq \ldots \geq d_m$ is non-increasing.
\end{proof}
Notice that any ample geodesic is strongly normal, and for real-analytic sub-Riemannian structures also the converse is true (see \cite[Prop. 3.12]{curvature}). The generic geodesic is ample and equiregular. More precisely, the set of points $x \in M$ such that there a exists non-empty Zariski open set $A_{x} \subseteq T_{x}^*M$ of initial covectors for which the associated geodesic is ample and equiregular with the same (maximal) growth vector, is open and dense in $M$. For more details, see \cite[Sec. 5]{lizel} and \cite[Sec. 5.2]{curvature}.

\subsubsection*{Young diagram}
For an ample, equiregular geodesic, the sequence of dimension stabilises, namely $\dim\DD_\gamma^m = \dim\DD_\gamma^{m+j} = n$ for $j\geq 0$, and we write $\mathcal{G}_\gamma = \{\dim\DD_\gamma^1,\ldots,\dim\DD_\gamma^m\}$. Thus, we associate with any ample, equiregular geodesic its Young diagram as follows. Recall that $d_{i}=\dim\DD_\gamma^i-\dim\DD_\gamma^{i-1}$ defines a decreasing sequence by Lemma~\ref{l:decreasing}. Then we can build a tableau $\y$ with $m$ columns of length $d_{i}$, for $i=1,\ldots,m$, as follows:
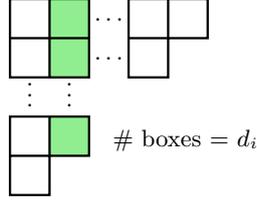
\begin{center}
\begin{tikzpicture}[x=0.26mm, y=0.26mm, inner xsep=0pt, inner ysep=0pt, outer xsep=0pt, outer ysep=0pt]
\path[line width=0mm] (61.05,70.00) rectangle +(127.98,100.00);
\draw(120.00,159.00) node[anchor=base]{\fontsize{9.39}{11.27}\selectfont $\ldots$};
\draw(120.00,139.00) node[anchor=base]{\fontsize{9.39}{11.27}\selectfont $\ldots$};
\draw(80.00,115.00) node[anchor=base]{\fontsize{9.39}{11.27}\selectfont $\vdots$};
\draw(100.00,115.00) node[anchor=base]{\fontsize{9.39}{11.27}\selectfont $\vdots$};
\definecolor{L}{rgb}{0,0,0}
\path[line width=0.30mm, draw=L] (130.00,170.00) -- (130.00,130.00) -- (150.00,130.00) -- (150.00,150.00) -- (170.00,150.00) -- (170.00,170.00) -- cycle;
\path[line width=0.30mm, draw=L] (150.00,170.00) -- (150.00,150.00);
\path[line width=0.30mm, draw=L] (130.00,150.00) -- (150.00,150.00);
\draw(122.00,95.00) node[anchor=base west]{\fontsize{8.54}{10.24}\selectfont \# boxes = $d_i$};
\definecolor{F}{rgb}{0.565,0.933,0.565}
\path[line width=0.30mm, draw=L, fill=F] (90.00,170.00) [rotate around={270:(90.00,170.00)}] rectangle +(40.00,20.00);
\path[line width=0.30mm, draw=L] (90.00,170.00) -- (70.00,170.00) -- (70.00,130.00) -- (90.00,130.00);
\path[line width=0.30mm, draw=L] (70.00,150.00) -- (110.00,150.00);
\path[line width=0.30mm, draw=L, fill=F] (90.00,110.00) [rotate around={270:(90.00,110.00)}] rectangle +(20.00,20.00);
\path[line width=0.30mm, draw=L] (90.00,70.00) [rotate around={90:(90.00,70.00)}] rectangle +(20.00,20.00);
\path[line width=0.30mm, draw=L] (70.00,90.00) -- (70.00,110.00) -- (90.00,110.00);
\end{tikzpicture}%

\end{center}
Indeed $\sum_{i=1}^m d_i = n=\dim M$ is the total number of boxes in $\y$. Let us discuss some examples.
\begin{example}
For a Riemannian structure, the flag of any non-trivial geodesic consists in a single space $\DD^1_{\gamma}(t) = T_{\gamma(t)}M$. Therefore $\mathcal{G}_\gamma(t) = \{n\}$ and all the geodesics are ample and equiregular. Roughly speaking, all the directions have the same (trivial) behaviour w.r.t. the Lie derivative.
\end{example}
\begin{example}\label{ex:contact}
Consider a contact, sub-Riemannian manifold with $\dim M = 2n+1$, and a non-trivial geodesic $\gamma$ with tangent field $\tanf \in \VecH$. Let $X_1,\ldots,X_{2n}$ be a local frame in a neighbourhood of the geodesic and $X_0$ the Reeb vector field. Let $\omega$ be the contact form. We define the invertible bundle map $J:\distr \to \distr$ by $\langle X|JY\rangle = d\omega(X,Y)$, for $X,Y \in \VecH$. Finally, we split $\distr = J\tanf \oplus J\tanf^\perp$ along the geodesic $\gamma(t)$. We obtain
\begin{equation}
\mc{L}_\tanf(Y) =\langle J\tanf|Y\rangle X_0 \mod \VecH, \qquad \forall\,Y \in \VecH.
\end{equation}
Therefore, the Lie derivative of fields in $J\tanf^\perp$ does not generate ``new directions''. On the other hand, $\mc{L}_\tanf(J\tanf) = X_0$ up to elements in $\VecH$. In this sense, the subspaces $J\tanf$ and $J\tanf^\perp$ are different w.r.t. Lie derivative: the former generates new directions, the latter does not. In the Young diagram, the subspace $J\tanf^\perp$ corresponds to the rectangular sub-diagram $\y_2$, while the subspace $J\tanf\oplus X_0$ corresponds to the rectangular sub-diagram $\y_1$ in Fig.~\ref{f:Yd}.b.
\end{example}

\begin{figure}[!ht]
\centering
\begin{tikzpicture}[x=0.30mm, y=0.30mm, inner xsep=0pt, inner ysep=0pt, outer xsep=0pt, outer ysep=0pt]
\definecolor{L}{rgb}{0,0,0}
\definecolor{F}{rgb}{0.565,0.933,0.565}
\path[line width=0.15mm, draw=L, fill=F] (-53.50,91.50) rectangle +(17.92,70.94);
\path[line width=0.15mm, draw=L] (-53.30,127.00) -- (-35.58,127.00);
\path[line width=0.15mm, draw=L] (-53.30,109.29) -- (-35.58,109.29);
\path[line width=0.15mm, draw=L] (-53.30,91.57) -- (-35.58,91.57);
\path[line width=0.15mm, draw=L] (-53.30,144.72) -- (-35.58,144.72);
\path[line width=0.15mm, draw=L, fill=F] (88.32,91.67) rectangle +(17.82,70.76);
\definecolor{F}{rgb}{1,0,0}
\path[line width=0.15mm, draw=L, fill=F] (88.43,144.72) rectangle +(35.43,17.72);
\path[line width=0.15mm, draw=L] (106.14,162.44) -- (106.14,144.72);
\path[line width=0.15mm, draw=L] (88.43,127.00) -- (106.14,127.00);
\path[line width=0.15mm, draw=L] (88.43,109.29) -- (106.14,109.29);
\path[line width=0.15mm, draw=L] (88.43,91.57) -- (106.14,91.57);
\definecolor{F}{rgb}{0.686,0.933,0.933}
\path[line width=0.15mm, draw=L, fill=F] (247.87,91.55) [rotate around={0:(247.87,91.55)}] rectangle +(17.71,17.73);
\definecolor{F}{rgb}{0.565,0.933,0.565}
\path[line width=0.15mm, draw=L, fill=F] (247.87,109.29) rectangle +(35.43,35.43);
\definecolor{F}{rgb}{1,0,0}
\path[line width=0.15mm, draw=L, fill=F] (247.87,144.72) rectangle +(70.86,17.72);
\path[line width=0.15mm, draw=L] (265.58,162.44) -- (265.58,109.29);
\path[line width=0.15mm, draw=L] (247.87,127.00) -- (283.30,127.00);
\path[line width=0.15mm, draw=L] (247.87,91.57) -- (265.58,91.57);
\path[line width=0.15mm, draw=L] (283.30,162.44) -- (283.30,144.72);
\path[line width=0.15mm, draw=L] (301.01,162.44) -- (301.01,144.72);
\path[line width=0.15mm, draw=L] (128.00,162.00) -- (133.00,162.00) -- (133.00,145.00) -- (128.00,145.00);
\path[line width=0.15mm, draw=L] (128.00,145.00) -- (133.00,145.00) -- (133.00,92.00) -- (128.00,92.00);
\draw(139.00,150.00) node[anchor=base west]{\fontsize{9.39}{11.27}\selectfont $\y_1$};
\draw(139.00,116.00) node[anchor=base west]{\fontsize{9.39}{11.27}\selectfont $\y_2$};
\draw(68.00,125.00) node[anchor=base west]{\fontsize{8.54}{10.24}\selectfont (b)};
\draw(227.00,125.00) node[anchor=base west]{\fontsize{8.54}{10.24}\selectfont (c)};
\draw(-75.00,124.00) node[anchor=base west]{\fontsize{8.54}{10.24}\selectfont (a)};
\end{tikzpicture}%
\caption{Young diagrams for (a) Riemannian, (b) contact, (c) a more general structure.}\label{f:Yd}
\end{figure}
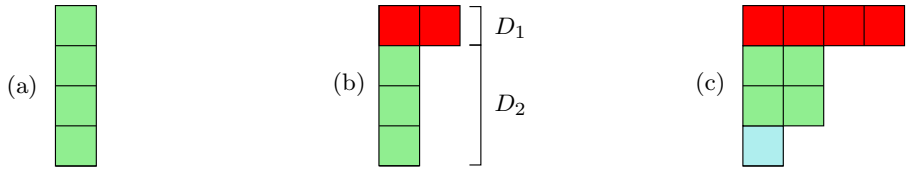

See Fig.~\ref{f:Yd} for some examples of Young diagrams. The number of boxes in the $i$-th row (i.e. $d_i$) is the number of new independent directions in $T_{\gamma(t)}M$ obtained by taking $(i-1)$-th Lie derivatives in the direction of $\tanf$.
\section{Jacobi fields revisited: conjugate points and Riccati equation}\label{s:Jac}

Let $\lambda \in T^*M$ be the covector associated with a strongly normal geodesic, projection of the extremal $\lambda(t) = e^{t\vec{H}}(\lambda)$. For any $\xi \in T_\lambda(T^*M)$ we define the field along the extremal $\lambda(t)$ as 
\begin{equation}
X(t):= e^{t\vec{H}}_* \xi \in T_{\lambda(t)}(T^*M).
\end{equation}
The set of vector fields obtained in this way is a $2n$-dimensional vector space, that we call \emph{the space of Jacobi fields along the extremal}. In the Riemannian case, the projection $\pi_*$ is an isomorphisms between the space of Jacobi fields along the extremal and the classical space of Jacobi fields along the geodesic $\gamma$. Thus, this definition is equivalent to the standard one in Riemannian geometry, does not need curvature or connection, and works for any strongly normal sub-Riemannian geodesic.

In Riemannian geometry, the study of one half of such a vector space, namely the subspace of classical Jacobi fields vanishing at zero, carries information about conjugate points along the given geodesic. By the aforementioned isomorphism, this corresponds to the subspace of Jacobi fields along the extremal such that $\pi_* X(0) = 0$. This motivates the following construction.

For any $\lambda\in T^*M$, let $\ve_{\lambda}:= \ker \pi_*|_\lambda \subset T_{\lambda}(T^*M)$ be the \emph{vertical subspace}. We define the family of Lagrangian subspaces along the extremal
\begin{equation}
\mc{L}(t):= e^{t\vec{H}}_* \ve_\lambda \subset T_{\lambda(t)}(T^*M).
\end{equation}

\begin{definition}
A time $t>0$ is a \emph{conjugate time} for $\gamma$ if $\mc{L}(t) \cap \ve_{\lambda(t)} \neq \{0\}$. Equivalently, we say that $\gamma(t) = \pi(\lambda(t))$ is a conjugate point w.r.t. $\gamma(0)$ along $\gamma(t)$. The \emph{first conjugate time} is the smallest conjugate time, namely $t_c(\gamma) = \inf\{t>0 \mid \mc{L}(t) \cap \ve_{\lambda(t)} \neq \{0\}\}$.
\end{definition}
Since the geodesic is strongly normal, the first conjugate time is separated from zero, namely there exists $\eps >0$ such that $\mc{L}(t) \cap \ve_{\lambda(t)} = \{0\}$ for all $t \in(0,\eps)$.
Notice that conjugate points correspond to the critical values of the sub-Riemannian exponential map with base in $\gamma(0)$. In other words, if $\gamma(t)$ is conjugate with $\gamma(0)$ along $\gamma$, there exists a one-parameter family of geodesics starting at $\gamma(0)$ and ending at $\gamma(t)$ at first order. Indeed, let $\xi \in \ve_\lambda$ such that $\pi_* \circ e^{t\vec{H}}_* \xi = 0$, then the vector field $\tau \mapsto \pi_* \circ e^{\tau\vec{H}}_* \xi$ is a classical Jacobi field along $\gamma$ which vanishes at the endpoints, and this is precisely the vector field of the aforementioned variation. 

In Riemannian geometry geodesics stop to be minimizing after the first conjugate time. This remains true for strongly normal sub-Riemannian geodesics (see, for instance, \cite{nostrolibro}). 


\subsection{Riemannian interlude}

In this section, we recall the concept of parallely transported frame along a geodesic in Riemannian geometry, and we give an equivalent characterisation in terms of a Darboux moving frame along the corresponding extremal lift.
Let $(M,\langle\cdot|\cdot\rangle)$ be a Riemannian manifold, endowed with the Levi-Civita connection $\nabla : \VecM \to \VecM$. In terms of a local orthonormal frame
\begin{equation}
\nabla_{X_j} X_i = \sum_{k=1}^n \Gamma_{ij}^k X_k, \qquad \Gamma_{ij}^k = \frac{1}{2}\left(c_{ij}^k + c_{ki}^j + c_{kj}^i\right),
\end{equation}
where $\Gamma_{ij}^k \in C^\infty(M)$ are the Christoffel symbols written in terms of the orthonormal frame. Notice that $\Gamma_{ij}^k = -\Gamma_{ik}^j$. 

Let $\gamma(t)$ be a geodesic and $\lambda(t)$ be the associated (normal) extremal, such that $\dot{\lambda}(t) = \vec{H}|_{\lambda(t)}$ and $\gamma(t) = \pi\circ\lambda(t)$. Let $\{X_1,\ldots,X_n\}$ a parallely transported frame along the geodesic $\gamma(t)$, i.e. $\nabla_{\dot{\gamma}} X_i = 0$.  Let $h_i:T^*M \to \R$ be the linear-on-fibers functions associated with $X_i$, defined by $h_i(\lambda):= \langle \lambda, X_i\rangle$. We define the (vertical) fields $\partial_{h_i}\in \text{Vec}(T^*M)$ such that $\partial_{h_i} (\pi^* g) = 0$, and $\partial_{h_i} (h_j) = \delta_{ij}$ for any $g \in C^\infty(M)$ and $i,j=1,\ldots,n$. We define a moving frame along the extremal $\lambda(t)$ as follows
\begin{equation}
E_i:=\partial_{h_i}, \qquad F_i := -[\vec{H},E_i],
\end{equation}
where the frame is understood to be evaluated at $\lambda(t)$. Notice that we can recover the parallely transported frame by projection, namely $\pi_* F_i|_{\lambda(t)} = X_i|_{\g(t)}$ for all $i$. In the following, for any vector field $Z$ along an extremal $\lambda(t)$ we employ the shorthand
\begin{equation}
\dot{Z}|_{\lambda(t)} := \left.\frac{d}{d\eps}\right|_{\eps = 0}e^{-\eps\vec{H}}_* Z|_{\lambda(t+\eps)} = [\vec{H},Z]|_{\lambda(t)}
\end{equation}
to denote the vector field along $\lambda(t)$ obtained by taking the Lie derivative in the direction of $\vec{H}$ of any smooth extension of $Z$. Notice that this is well defined, namely its value at $\lambda(t)$ does not depend on the choice of the extension. We state the properties of the moving frame in the following proposition.
\begin{proposition}\label{p:riemcan}
The smooth moving frame $\{E_i,F_i\}_{i=1}^n$ has the following properties:
\begin{itemize}
\item[(i)] $\spn\{E_i|_{\lambda(t)}\} = \ve_{\lambda(t)}$.
\item[(ii)] It is a Darboux basis, namely
\begin{equation}
\sigma(E_i,E_j) = \sigma(F_i,F_j) = \sigma(E_i,F_j) - \delta_{ij} = 0, \qquad i,j=1,\ldots,n.
\end{equation}
\item[(iii)] The frame satisfies structural equations
\begin{equation}
\dot{E}_i = - F_i, \qquad \dot{F}_i = \sum_{j=1}^n R_{ij}(t) E_j,
\end{equation}
for some smooth family of $n\times n$ symmetric matrices $R(t)$.
\end{itemize}
Properties (i)-(iii) uniquely define the moving frame up to orthogonal transformations. More precisely if $\{\wt{E}_i,\wt{F}_j\}_{i=1}^n$ is another smooth moving frame along $\lambda(t)$ satisfying (i)-(iii), with some matrix $\wt{R}(t)$ then there exist a constant, orthogonal matrix $O$ such that 
\begin{equation}\label{eq:orthonormal}
\wt{E}_i|_{\lambda(t)} = \sum_{j=1}^n O_{ij}E_j|_{\lambda(t)}, \qquad  \wt{F}_i|_{\lambda(t)} = \sum_{j=1}^nO_{ij}F_j|_{\lambda(t)}, \qquad \wt{R}(t) = O R(t) O^*. 
\end{equation}
\end{proposition}
A few remarks are in order. Property (ii) implies that $\spn\{E_1,\ldots,E_n\}$, $\spn\{F_1,\ldots,F_n\}$, evaluated at $\lambda(t)$, are Lagrangian subspaces of $T_{\lambda(t)}(T^*M)$. Eq.~\eqref{eq:orthonormal} reflects the fact that a parallely transported frame is defined up to constant orthogonal transformations. In particular, one could use properties (i)-(iii) to \emph{define} the parallel transport along $\gamma(t)$ by $X_i|_{\gamma(t)}:=\pi_* F_i|_{\lambda(t)}$. Finally, the symmetric matrix $R(t)$ induces a well defined quadratic form $\mathfrak{R}_{\gamma(t)}:T_{\gamma(t)}M \to \R$
\begin{equation}\label{eq:defff}
\RR_{\gamma(t)}(v) := \sum_{i,j=1}^n R_{ij}(t) v_i v_j, \qquad v = \sum_{i=1}^n v_i X_i|_{\gamma(t)} \in T_{\gamma(t)}M.
\end{equation}
Indeed Proposition~\ref{p:riemcan} implies that the definition of $\RR_{\gamma(t)}$ does not depend on the choice of the parallely transported frame.

\begin{lemma}\label{l:Riemanncurv}
Let $R^\nabla : \VecM \times \VecM \times \VecM \to \VecM$ the Riemannian curvature tensor w.r.t. the Levi-Civita connection. Then
\begin{equation}
\RR_{\gamma}(v) = \langle R^\nabla(v,\dot{\gamma})\dot{\gamma}|v\rangle, \qquad v \in T_{\gamma}M,
\end{equation}
where we suppressed the explicit dependence on time. 
\end{lemma}

In other words, for any unit vector $v \in T_{\gamma}M$, $\RR_{\gamma}(v) = \Sec(v,\dot{\gamma})$ is the sectional curvature of the plane generated by $v$ and $\dot{\gamma}$, i.e. the \emph{directional curvature} in the direction of the geodesic.
The proof of Proposition~\ref{p:riemcan} and Lemma~\ref{l:Riemanncurv} can be found in Appendix~\ref{a:riemcan}.

\subsection{Canonical frame}\label{s:can}

The concept of Levi-Civita connection and covariant derivative is not available for general sub-Riemannian structures, and it is not clear how to parallely transport a frame along a sub-Riemannian geodesic. Nevertheless, in \cite{lizel}, the authors introduce a parallely transported frame along the corresponding extremal $\lambda(t)$ which, in the spirit of Proposition~\ref{p:riemcan}, generalises the concept of parallel transport also to (sufficiently regular) sub-Riemannian extremals.

Consider an ample, equiregular geodesic, with Young diagram $\y$, with $k$ rows, of length $n_1,\ldots,n_k$. Indeed $n_1+\ldots+n_k = n$. The moving frame we are going to introduce is indexed by the boxes of the Young diagram, so we fix some terminology first. Each box is labelled ``$ai$'', where $a=1,\ldots,k$ is the row index, and $i=1,\ldots,n_a$ is the progressive box number, starting from the left, in the specified row. Briefly, the notation $ai \in \y$ denotes the generic box of the diagram. We employ letters from the beginning of the alphabet $a,b,c,\dots$ for rows, and letters from the middle of the alphabet $i,j,h,\dots$ for the position of the box in the row.  

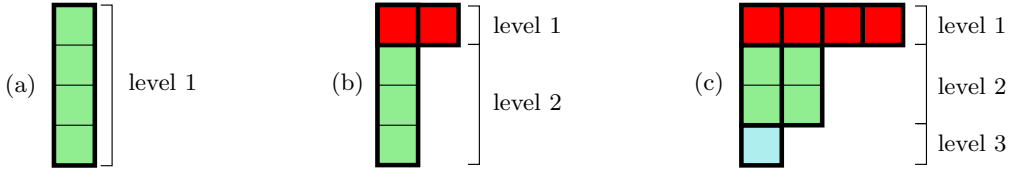
\begin{figure}[h]
\centering
\begin{tikzpicture}[x=0.30mm, y=0.30mm, inner xsep=0pt, inner ysep=0pt, outer xsep=0pt, outer ysep=0pt]
\path[line width=0mm] (-75.00,91.40) rectangle +(300.63,71.04);
\definecolor{L}{rgb}{0,0,0}
\definecolor{F}{rgb}{0.565,0.933,0.565}
\path[line width=0.60mm, draw=L, fill=F] (-53.50,91.50) rectangle +(17.92,70.94);
\path[line width=0.15mm, draw=L] (-53.30,127.00) -- (-35.58,127.00);
\path[line width=0.15mm, draw=L] (-53.30,109.29) -- (-35.58,109.29);
\path[line width=0.15mm, draw=L] (-53.30,91.57) -- (-35.58,91.57);
\path[line width=0.15mm, draw=L] (-53.30,144.72) -- (-35.58,144.72);
\path[line width=0.60mm, draw=L, fill=F] (88.32,91.67) rectangle +(17.82,70.76);
\definecolor{F}{rgb}{1,0,0}
\path[line width=0.60mm, draw=L, fill=F] (88.43,144.72) rectangle +(35.43,17.72);
\path[line width=0.60mm, draw=L] (106.14,162.44) -- (106.14,144.72);
\path[line width=0.15mm, draw=L] (88.43,127.00) -- (106.14,127.00);
\path[line width=0.15mm, draw=L] (88.43,109.29) -- (106.14,109.29);
\path[line width=0.15mm, draw=L] (88.43,91.57) -- (106.14,91.57);
\definecolor{F}{rgb}{0.686,0.933,0.933}
\path[line width=0.60mm, draw=L, fill=F] (247.87,91.55) [rotate around={0:(247.87,91.55)}] rectangle +(17.71,17.73);
\definecolor{F}{rgb}{0.565,0.933,0.565}
\path[line width=0.60mm, draw=L, fill=F] (247.87,109.29) rectangle +(35.43,35.43);
\definecolor{F}{rgb}{1,0,0}
\path[line width=0.60mm, draw=L, fill=F] (247.87,144.72) rectangle +(70.86,17.72);
\path[line width=0.60mm, draw=L] (265.58,162.44) -- (265.58,109.29);
\path[line width=0.15mm, draw=L] (247.87,127.00) -- (283.30,127.00);
\path[line width=0.15mm, draw=L] (247.87,91.57) -- (265.58,91.57);
\path[line width=0.60mm, draw=L] (283.30,162.44) -- (283.30,144.72);
\path[line width=0.60mm, draw=L] (301.01,162.44) -- (301.01,144.72);
\draw(-20.85,124.71) node[anchor=base west]{\fontsize{9.39}{11.27}\selectfont level 1};
\path[line width=0.15mm, draw=L] (128.00,162.00) -- (133.00,162.00) -- (133.00,145.00) -- (128.00,145.00);
\path[line width=0.15mm, draw=L] (128.00,145.00) -- (133.00,145.00) -- (133.00,92.00) -- (128.00,92.00);
\draw(139.00,150.00) node[anchor=base west]{\fontsize{9.39}{11.27}\selectfont level 1};
\draw(139.00,116.00) node[anchor=base west]{\fontsize{9.39}{11.27}\selectfont level 2};
\path[line width=0.15mm, draw=L] (324.00,162.00) -- (329.00,162.00) -- (329.00,145.00) -- (324.00,145.00);
\draw(334.00,150.00) node[anchor=base west]{\fontsize{9.39}{11.27}\selectfont level 1};
\path[line width=0.15mm, draw=L] (324.00,145.00) -- (329.00,145.00) -- (329.00,110.00) -- (324.00,110.00);
\draw(334.00,123.00) node[anchor=base west]{\fontsize{9.39}{11.27}\selectfont level 2};
\path[line width=0.15mm, draw=L] (325.00,110.00) -- (329.00,110.00) -- (329.00,92.00) -- (324.00,92.00);
\draw(334.00,98.00) node[anchor=base west]{\fontsize{9.39}{11.27}\selectfont level 3};
\draw(68.00,125.00) node[anchor=base west]{\fontsize{8.54}{10.24}\selectfont (b)};
\draw(227.00,125.00) node[anchor=base west]{\fontsize{8.54}{10.24}\selectfont (c)};
\draw(-75.00,124.00) node[anchor=base west]{\fontsize{8.54}{10.24}\selectfont (a)};
\path[line width=0.15mm, draw=L] (-33.00,162.40) -- (-28.00,162.40) -- (-28.00,91.40) -- (-33.00,91.40);
\end{tikzpicture}%
\caption{Levels (shaded regions) and superboxes (delimited by bold lines) for the Young diagram of (a) Riemannian, (b) contact, (c) a more general structure. The Young diagram for any Riemannian geodesic has a single level and a single superbox. The Young diagram of any contact sub-Riemannian geodesic has levels two levels containing $2$ and $1$ superboxes, respectively. The Young diagram (c) has three levels with 4, 2, 1 superboxes, respectively.}\label{f:Yd2}
\end{figure}
We collect the rows with the same length in $\y$, and we call them \emph{levels} of the Young diagram.  In particular, a level is the union of $r$ rows $\y_1,\ldots,\y_r$, and $r$ is called the \emph{size} of the level. The set of all the boxes $ai \in\y$ that belong to the same column and the same level of $\y$ is called \emph{superbox}. We use greek letters $\alpha,\beta,\dots$ to denote superboxes. Notice that that two boxes $ai$, $bj$ are in the same superbox if and only if $ai$ and $bj$ are in the same column of $\y$ and in possibly distinct row but with same length, i.e. if and only if $i=j$ and $n_a = n_b$. See Fig.~\ref{f:Yd2} for examples of levels and superboxes for Riemannian, contact and more general structures.

\begin{theorem}[See \cite{lizel}]\label{p:can}
There exists a smooth moving frame $\{E_{ai},F_{ai}\}_{ai \in \y}$ along the extremal $\lambda(t)$ such that
\begin{itemize}
\item[(i)] $\spn\{E_{ai}|_{\lambda(t)}\} = \ve_{\lambda(t)}$.
\item[(ii)] It is a Darboux basis, namely
\begin{equation}
\sigma(E_{ai},E_{bj}) = \sigma(F_{ai},F_{bj}) = \sigma(E_{ai},F_{bj}) = \delta_{ab}\delta_{ij}, \qquad ai,bj \in \y.
\end{equation}
\item[(iii)] The frame satisfies structural equations
\begin{equation}\label{zelframe}
\displaystyle\begin{cases}	
\dot{E}_{ai} = E_{a(i-1)} & a = 1,\dots,k,\quad i = 2,\dots, n_a,\\[0.1cm]
\dot{E}_{a1} = -F_{a1} & a= 1,\dots,k, \\[0.1cm]
\dot{F}_{ai} = \sum_{bj \in \y} R_{ai,bj}(t) E_{bj} - F_{a(i+1)} & a=1,\dots,k,\quad i = 1,\dots,n_a-1,\\[0.1cm]
\dot{F}_{an_a} = \sum_{bj \in \y} R_{bj,an_a}(t) E_{bj}  & a = 1, \dots,k,
\end{cases}
\end{equation}
for some smooth family of $n\times n$ symmetric matrices $R(t)$, with components $R_{ai,bj}(t) = R_{bj,ai}(t)$, indexed by the boxes of the Young diagram $\y$. The matrix $R(t)$ is \emph{normal} in the sense of \cite{lizel}. 
\end{itemize}
Properties (i)-(iii) uniquely define the frame up to orthogonal transformation that preserve the Young diagram. More precisely, if $\{\wt{E}_{ai},\wt{F}_{ai}\}_{ai \in \y}$ is another smooth moving frame along $\lambda(t)$ satisfying i)-iii), with some normal matrix $\wt{R}(t)$, then for any superbox $\alpha$ of size $r$ there exists an orthogonal (constant) $r\times r $ matrix $O^\alpha$ such that
\begin{equation}
\wt{E}_{ai} = \sum_{bj \in \alpha} O^\alpha_{ai,bj} E_{bj}, \qquad \wt{F}_{ai} = \sum_{bj \in \alpha} O^\alpha_{ai,bj} F_{bj}, \qquad ai \in \alpha.
\end{equation}
\end{theorem}
Theorem~\ref{p:can} implies that the following objects are well defined:
\begin{itemize}
\item The scalar product $\langle\cdot|\cdot\rangle_{\gamma(t)}$, depending on $\gamma(t)$, such that the fields $X_{ai}|_{\gamma(t)}:=\pi_*F_{ai}|_{\lambda(t)}$ along $\gamma(t)$ are an orthonormal frame.
\item A splitting of $T_{\gamma(t)} M$, orthogonal w.r.t. $\langle\cdot|\cdot\rangle_{\gamma(t)}$
\begin{equation}
T_{\gamma(t)}M = \bigoplus_{\alpha}S_{\gamma(t)}^{\alpha}, \qquad S_{\gamma(t)}^{\alpha}:=\spn\{ X_{ai}|_{\gamma(t)}\mid \, ai \in \alpha\},
\end{equation}
where the sum is over the superboxes $\alpha$ of $\y$. Notice that the dimension of $S_{\gamma(t)}^{\alpha}$ is equal to the size $r$ of the level in which the superbox $\alpha$ is contained.
\item The sub-Riemannian directional curvature, defined as the quadratic form $\RR_{\gamma(t)}: T_{\gamma(t)} M \to \R$ whose representative matrix, in terms of an orthonormal frame $\{X_{ai}\}_{ai \in \y}$ is $R_{ai,bj}(t)$.
\item For each superbox $\alpha$, the \emph{sub-Riemannian Ricci curvatures}
\begin{equation}
\Ric_{\gamma(t)}^\alpha:= \trace \left(\RR_{\gamma(t)}\big|_{S^{\alpha}_{\gamma(t)}}\right)=\sum_{ai \in \alpha} \RR_{\gamma(t)}(X_{ai}),
\end{equation}
which is precisely the partial trace of $\RR_{\gamma(t)}$, identified through the scalar product with an operator on $T_{\gamma(t)}M$, on the subspace $S^\alpha_{\gamma(t)} \subseteq T_{\gamma(t)}M$.
\end{itemize}

In this sense, each superbox $\alpha$ in the Young diagram corresponds to a well defined subspace $S_{\gamma(t)}^\alpha$ of $T_{\gamma(t)}M$. Notice that, for Riemannian structures, the Young diagram is trivial with $n$ rows of length $1$, there is a single superbox, Theorem~\ref{p:can} reduces to Proposition~\ref{p:riemcan}, the scalar product $\langle\cdot|\cdot\rangle_{\gamma(t)}$ reduces to the Riemannian product computed along the geodesic $\gamma(t)$, the orthogonal splitting is trivial, the directional curvature $\RR_{\gamma(t)} = \Sec(\dot\gamma,\cdot)$ is the sectional curvature of the planes containing $\dot\gamma(t)$ and there is only one Ricci curvature $\Ric_{\gamma(t)} = \mathrm{Ric}^\nabla(\dot{\gamma}(t))$, where $\mathrm{Ric}^\nabla : \VecM \to \R$ is the classical Ricci curvature.

\subsubsection*{A compact form for the structural equations}
We rewrite system \eqref{zelframe} in a compact form. In the sequel it will be convenient to split a frame $\{E_{ai},F_{ai}\}_{ai \in \y}$ in subframes, relative to the rows of the Young diagram. For $a=1,\dots,k$, the symbol $E_a$ denotes the $n_a$-dimensional row vector
\begin{equation}
E_a = (E_{a1},E_{a2},\dots,E_{an_a}),
\end{equation}
with  analogous notation for $F_a$. Similarly, $E$ denotes the $n$-dimensional row vector
\begin{equation}
E = (E_1,\dots,E_k),
\end{equation}
and similarly for $F$.
Let $\Gamma_1 = \Gamma_1(\y),\Gamma_2=\Gamma_2(\y)$ be $n\times n$ matrices, depending on the Young diagram $\y$, defined as follows: for $a,b = 1,\dots,k$, $i=1,\dots,n_a$, $j=1,\dots,n_b$, we set
\begin{gather}
(\Gamma_1)_{ai,bj} := \delta_{ab}\delta_{i,j-1}, \label{eq:G1}\\
(\Gamma_2)_{ai,bj} := \delta_{ab}\delta_{i1}\delta_{j1} \label{eq:G2} .
\end{gather}
It is convenient to see $\Gamma_1$ and $\Gamma_2$ as block diagonal matrices, the $a$-th block on the diagonal being a $n_a\times n_a$ matrix with components $\delta_{i,j-1}$ and $\delta_{i1}\delta_{j1}$, respectively (see also Eq.~\eqref{eq:Gamma}). Notice that $\Gamma_1$ is nilpotent and $\Gamma_2$ is idempotent.  Then, we rewrite the system \eqref{zelframe} as follows
\begin{equation}
\begin{pmatrix}
\dot{E} & \dot{F}
\end{pmatrix} = 
\begin{pmatrix}
E & F
\end{pmatrix}
\begin{pmatrix}
\Gamma_1 & R(t) \\ -\Gamma_2 & -\Gamma_1^*
\end{pmatrix}.
\end{equation}
By exploiting the structural equations, we write a linear differential equation in $\R^{2n}$ that rules the evolution of the Jacobi fields along the extremal.

\subsection{Linearized Hamiltonian}
Let $\xi \in T_{\lambda}(T^*M)$ and $X(t):= e^{t\vec{H}}_* \xi$ be the associated Jacobi field along the extremal. In terms of any moving frame $\{E_{ai},F_{ai}\}_{ai \in \y}$ along $\lambda(t)$, it has components $(p(t),x(t)) \in \R^{2n}$, namely 
\begin{equation}
X(t) = \sum_{ai \in \y} p_{ai}(t) E_{ai}|_{\lambda(t)} + x_{ai}(t) F_{ai}|_{\lambda(t)}.
\end{equation}
If we choose the canonical frame, using the structural equations, we obtain that the coordinates of the Jacobi field satisfy the following system of linear ODEs:
\begin{equation}\label{eq:linsys1}
\begin{pmatrix}
\dot{p} \\ \dot{x}
\end{pmatrix} = \begin{pmatrix} - \Gamma_1 & -R(t) \\ \Gamma_2 & \Gamma_1^*
\end{pmatrix} \begin{pmatrix}
p \\ x
\end{pmatrix}.
\end{equation}
In this sense, the canonical frame is a tool to write the linearisation of the Hamiltonian flow along the geodesic in a canonical form. The r.h.s. of Eq.~\eqref{eq:linsys1} is the ``linearised Hamiltonian vector field'', written in its normal form (see also Eq.~ \eqref{eq:Hamiltonian}). The linearised Hamiltonian field is, in general, non-autonomous. Notice also that the canonical form of the linearisation depends on the Young diagram $\y$ (through the matrices $\Gamma_1$ and $\Gamma_2$) and the curvature matrix $R(t)$.

In the Riemannian case, $\y = \ytableaushort{\empty}$ for any geodesic, $\Gamma_1 =0$, $\Gamma_2 = \mathbb{I}$ and we recover the classical Jacobi equation, written in terms of an orthonormal frame along the geodesic
\begin{equation}
\ddot{x}+R(t)x = 0.
\end{equation}

\subsection{Riccati equation: blow-up time and conjugate time}
Now we study, with a single matrix equation, the space of Jacobi fields along the extremal associated with an ample, equiregular geodesic. We write the generic element of $\mc{L}(t)$ in terms of the frame along the extremal. Let $E_{\lambda(t)}, F_{\lambda(t)}$ be row vectors, whose entries are the elements of the frame. The action of $e^{t\vec{H}}_*$ is meant entry-wise. Then
\begin{equation}
\mc{L}(t) \ni e^{t\vec{H}}_* E_{\lambda(0)} = E_{\lambda(t)} M(t) + F_{\lambda(t)} N(t),
\end{equation}
for some smooth families $M(t), N(t)$ of $n\times n$ matrices. Notice that
\begin{equation}
M(0) = \mathbb{I}, \qquad N(0) = 0, \qquad \det{N}(t) \neq 0 \text{ for } t \in (0,\eps).
\end{equation}
The first $t>0$ such that $\det{N}(t) = 0$ is indeed the first conjugate time. By using once again the structural equations, we obtain the following system of linear ODEs:
\begin{equation}\label{eq:linsys}
\frac{d}{dt}\begin{pmatrix}
M \\ N
\end{pmatrix} = 
\begin{pmatrix}
-\Gamma_1 & -R(t) \\
\Gamma_2 & \Gamma_1^*
\end{pmatrix} \begin{pmatrix}
M \\ N
\end{pmatrix}.
\end{equation}
The solution of the Cauchy problem with the initial datum $M(0) = \mathbb{I}$, $N(0) = 0$ is defined on the whole interval on which $R(t)$ is defined. The columns of the $2n\times n$ matrix $\bigl( \begin{smallmatrix}M \\ N\end{smallmatrix}\bigr)$ are the components of Jacobi fields along the extremal w.r.t. the given frame, and they generate the $n$-dimensional subspace of Jacobi fields $X(t)$ along the extremal $\lambda(t)$ such that $\pi_* X(0) = 0$.

Since, for small $t>0$, $\mc{L}(t) \cap \mc{L}(0) = \{0\}$, we have that
\begin{equation}
\mc{L}(t) = \spn\{F_{\lambda(t)} + E_{\lambda(t)} V(t)\},\qquad t>0,
\end{equation}
where $V(t):=M(t)N(t)^{-1}$ is well defined and smooth for $t>0$ until the first conjugate time. Since $\mc{L}(t)$ is a Lagrangian subspace and the canonical frame is Darboux, $V(t)$ is a symmetric matrix. Moreover it satisfies the following Riccati equation:
\begin{equation}\label{eq:Riccati}
\dot{V} = -\Gamma_1 V - V \Gamma_1^* - R(t) - V \Gamma_2 V.
\end{equation}
We characterize $V(t)$ as the solution of a Cauchy problem with limit initial condition.
\begin{lemma}
The matrix $V(t)$ is the unique solution of the Cauchy problem
\begin{equation}\label{eq:cauchyp}
\dot{V} = -\Gamma_1 V - V \Gamma_1^* - R(t) - V \Gamma_2 V, \qquad \displaystyle\lim_{t\to 0^+} V^{-1} = 0,
\end{equation}
in the sense that $V(t)$ is the unique solution such that $V(t)$ is invertible for small $t>0$ and $\lim_{t\to 0^+}V(t)^{-1} = 0$.
\end{lemma}
\begin{proof}
As we already observed, $V(t)$ satisfies Eq.~\eqref{eq:Riccati}. Moreover $V(t)$ is invertible for $t>0$ small enough, $V(t)^{-1} = N(t) M(t)^{-1}$ and $\lim_{t\to 0^+} V^{-1} = 0$. The uniqueness follows from the well-posedness of the limit Cauchy problem. See Lemma~\ref{l:limit} in Appendix \ref{a:well}.
\end{proof}

It is well known that the solutions of Riccati equations are not, in general, defined for all $t$, but they may blow up at finite time. The next proposition relates the occurrence of such blow-up time with the first conjugate point along the geodesic.

\begin{proposition}\label{p:blowconj}
Let $V(t)$ the unique solution of \eqref{eq:cauchyp}, defined on its maximal interval $I \subseteq (0,+\infty)$.
 Let $t_c := \inf\{t >0|\, \mc{L}(t) \cap \ve_{\lambda(t)} \neq \{0\}\}$ be the first conjugate point along the geodesic. Then $I= (0,t_c)$.
\end{proposition}
\begin{proof}
First, we prove that $ I \supseteq (0,t_c)$. For any $t \in (0,t_c)$, $\mc{L}(t)$ is transversal to $\ve_{\lambda(t)}$. Then the matrix $N(t)$ is non-degenerate for all $t \in (0,t_c)$. Then $V(t):= M(t)N(t)^{-1}$ is the solution of \eqref{eq:cauchyp}, and $I\supseteq (0,t_c)$.

By contradiction, assume that $I \supset (0,t_c)$. Consider the $n$-dimensional smooth families of subspaces of $T_{\lambda(t)}(T^*M)$:
\begin{align}
\mathcal{L}(t)& =\spn\{F_{\lambda(t)}N(t) +E_{\lambda(t)}M(t)\}, & & t \in [0,+\infty) ,\\
\wt{\mathcal{L}}(t) & = \spn\{F_{\lambda(t)} + E_{\lambda(t)}V(t)\}, & &  t \in I.
\end{align}
Observe that $\wt{\mathcal{L}}(t) \cap \ve_{\lambda(t)} = \{0\}$ for all $t \in I$. On $(0,t_c)$ we have $V(t) = M(t) N(t)^{-1}$, hence $\mathcal{L}(t) = \wt{\mathcal{L}}(t)$ on this interval. By continuity, also $\mathcal{L}(t_c) =\wt{\mathcal{L}}(t_c)$. But the first subspace intersects $\mathcal{V}_{\lambda(t_c)}$ (by definition of $t_c$), while the second does not (by construction).
\end{proof}

Proposition~\ref{p:blowconj} states that the problem of finding the first conjugate time is equivalent to the study of the blow-up time of the Cauchy problem \eqref{eq:cauchyp} for the Riccati equation.
%
\section{Microlocal comparison theorem}\label{s:microlocal}

In Sec.~\ref{s:Jac}, we reduced the problem of finding the conjugate points along an ample, equiregular sub-Riemannian geodesic to the study of the blow-up time of the solution of the Cauchy problem
\begin{equation}
\dot{V} + \Gamma_1 V + V \Gamma_1^* + R(t) + V \Gamma_2 V = 0, \qquad \displaystyle\lim_{t\to 0^+} V^{-1} = 0.
\end{equation}
It is well known that the same equation controls the conjugate times of a LQ optimal control problems, defined by appropriate matrices $A,B,Q$, where $A=\Gamma_1^*$, $BB^* = \Gamma_2$, and the potential $Q$  replaces $R(t)$. In this sense, for what concerns the study of conjugate points, LQ problems represent the natural \emph{constant curvature models}.

\subsection{LQ optimal control problems}

Linear quadratic optimal control problems (LQ in the following) are a classical topic in control theory. They consist in a linear control system with a cost given by a quadratic Lagrangian. We briefly recall the general features of a LQ problem, and we refer to \cite[Ch. 16]{agrachevbook} and \cite[Ch. 7]{jurdjevicbook} for further details. We are interested in \emph{admissible trajectories}, namely  curves $x:[0,t]\to \mathbb{R}^n$ for which there exists a control $u \in L^2([0,t],\mathbb{R}^k)$ such that
\begin{equation}\label{eq:lq1}
\dot{x} = Ax +Bu, \qquad x(0) = x_0, \qquad x(t) = x_1,\qquad x_0,x_1,t \text{ fixed},
\end{equation}
that minimize a quadratic functional $\phi_{t}: L^2([0,t],\mathbb{R}^k) \to \mathbb{R}$ of the form
\begin{equation}\label{eq:lq2}
\phi_{t}(u) = \frac{1}{2}\int_{0}^{t} \left(u^* u - x^*Qx \right)dt.
\end{equation}
Here $A,B,Q$ are constant matrices of the appropriate dimension. The vector $Ax$ represents the \emph{drift}, while the columns of $B$ are the controllable directions. The meaning of the \emph{potential} term $Q$ will be clear later, when we will introduce the Hamiltonian of the LQ problem.

We only deal with \emph{controllable} systems, i.e. we assume that there exists $m >0$ such that
\begin{equation}
\rank(B,AB,\ldots,A^{m-1}B) = n.
\end{equation}
This hypothesis implies that, for any choice of $t,x_0,x_1$, the set of controls $u$ such that the associated trajectory $x_u :[0,t] \to \mathbb{R}^n$ connects $x_0$ with $x_1$ in time $t$ is not empty.

It is well known that the optimal trajectories of the LQ system are projections $(p,x) \mapsto x$ of the solutions of the Hamiltonian system
\begin{equation}
\dot{p}  = -\partial_x H (p,x), \qquad \dot{x} = \partial_p H (p,x), \qquad (p,x) \in T^*\R^n = \R^{2n},
\end{equation}
where the Hamiltonian function $H: \mathbb{R}^{2n} \to \mathbb{R}$ is defined by
\begin{equation}\label{eq:Hamiltonian}
H(p,x) = \frac{1}{2}\begin{pmatrix} p^* & x^* \end{pmatrix} \begin{pmatrix}
BB^* & A \\ A^* & Q
\end{pmatrix} \begin{pmatrix}
p \\ x
\end{pmatrix}.
\end{equation}
We denote by $P_t : \R^{2n} \to \R^{2n}$ the flow of the Hamiltonian system, which is defined for all $t \in \R$. We employ canonical coordinates $(p,x)$ on $T^*\R^n = \R^{2n}$ such that the symplectic form is written $\sigma = \sum_{i=1}^n dp_i \wedge dx_i$. The flow lines of $P_t$ are the integral lines of the \emph{Hamiltonian vector field} $\vec{H} \in \text{Vec}(\mathbb{R}^{2n})$, defined by $dH(\cdot) = \sigma(\,\cdot\,,\vec{H})$. More explicitly
\begin{equation}\label{eq:Hamiltonianvf}
\vec{H}_{(p,x)}=\begin{pmatrix} -A^* & -Q \\ BB^* & A \end{pmatrix}\begin{pmatrix}
p \\ x
\end{pmatrix}.
\end{equation}

We stress that not all the integral lines of the Hamiltonian flow lead to minimizing solutions of the LQ problem, since they only satisfy first order conditions for optimality. Sufficiently short segments, however, are optimal, but they lose optimality at some time $t>0$, called the \emph{first conjugate time}.
\begin{definition}
We say that $t$ is a conjugate time if there exists a solution of the Hamiltonian equations such that $x(0) = x(t) = 0$.
\end{definition}

The first conjugate time determines existence and uniqueness of minimizing solutions of the LQ problem, as specified by the following proposition (see \cite[Sec. 16.4]{agrachevbook}).

\begin{proposition}
Let $t_c$ be the first conjugate time of the LQ problem \eqref{eq:lq1}-\eqref{eq:lq2}
\begin{itemize}
\item For $t<t_c$, for any $x_0,x_1$ there exists a unique minimizer connecting $x_0$ with $x_1$ in time $t$. 
\item For $t>t_c$, for any $x_0,x_1$ there exists no minimizer connecting $x_0$ with $x_1$ in time $t$.
\end{itemize}
\end{proposition}

The first conjugate time can be also characterised in terms of blow-up time of a matrix Riccati equation. Consider the vector subspace of solutions of Hamilton equations such that $x(0) = 0$. A basis of such a space is given by the solutions $(p_i(t),x_i(t))$ with initial condition $p_i(0) := e_i$, $x_i(0) = 0$, where $e_i$, for $i=1,\ldots,n$ is the standard basis of $\R^n$. Consider the matrices $M$, $N$, whose columns are the vectors $p_i(t)$ and $x_i(t)$, respectively. They solve the following equation:
\begin{equation}
\frac{d}{dt}\begin{pmatrix}
M \\ N
\end{pmatrix} = 
\begin{pmatrix}
-A^* & -Q \\
BB^* & A
\end{pmatrix} \begin{pmatrix}
M \\ N
\end{pmatrix},
\end{equation}
where $M(0) = \mathbb{I}$ and $N(0) = 0$. Under the controllability condition, $N(t)$ is non-singular for $t>0$ sufficiently small. By definition, the first conjugate time of the LQ problem is the first $t>0$ such that $N(t)$ is singular. Thus, consider $V(t) :=M(t)N(t)^{-1}$. The matrix $V(t)$ is symmetric and is the unique solution of the following Cauchy problem with limit initial condition:
\begin{equation}\label{eq:Ricintro}
\dot{V} + A^*V + VA + Q + VBB^*V = 0, \qquad \lim_{t\to 0^+} V^{-1} = 0.
\end{equation}
Thus we have the following characterization of the first conjugate time of the LQ  problem.
\begin{lemma}
The maximal interval of definition of the unique solution of the Cauchy problem
\begin{equation}
\dot{V} + A^*V + VA + Q + VBB^*V = 0, \qquad \lim_{t\to 0^+} V^{-1} = 0,
\end{equation}
is $I= (0,t_c)$, where $t_c$ is the first conjugate time of the associated LQ optimal control problem.
\end{lemma}
The same characterisation holds also for conjugate points along sub-Riemannian geodesics (see Proposition~\ref{p:blowconj}), and in this sense LQ problems provide models for computing conjugate times along sub-Riemannian geodesics.

\subsection{Constant curvature models}

Let $\y$ be a Young diagram associated with some ample, equiregular geodesic, and let $\Gamma_1 = \Gamma_1(\y)$, $\Gamma_2 = \Gamma_2(\y)$ the matrices defined in Sec.~\ref{s:prel}. Let $Q$ be a symmetric $n\times n$ matrix.

\begin{definition}
We denote by $\mathrm{LQ}(\y;Q)$ the \emph{constant curvature model}, associated with a Young diagram $\y$ and constant curvature equal to $Q$, defined by the LQ problem with Hamiltonian
\begin{equation}
H(p,x) = \frac{1}{2}\left(p^* BB^* p + 2p^*Ax + x^*Qx\right), \qquad A = \Gamma_1^*, \quad BB^* = \Gamma_2.
\end{equation}
We denote by $t_c(\y;Q) \leq +\infty$ the first conjugate time of $\mathrm{LQ}(\y;Q)$.
\end{definition}
\begin{remark}
Indeed there are many matrices $B$ such that $BB^* = \Gamma_2$, namely LQ problems with the same Hamiltonian, but their first conjugate time is the same. In particular, without loss of generality, one may choose $B = BB^* = \Gamma_2$.
\end{remark}

In general, it is not trivial to deduce whether $t_c(\y;Q) < +\infty$ or not, and this will be crucial in our comparison theorems. Nevertheless we have the following result in terms of the representative matrix of the Hamiltonian vector field $\vec{H}$ given by Eq.~\eqref{eq:Hamiltonianvf} (see \cite{ARS}).

\begin{theorem}\label{t:ARS}
The following dichotomy holds true for a controllable LQ optimal control system:
\begin{itemize}
\item If $\vec{H}$ has at least one odd-dimensional Jordan block corresponding to a pure imaginary eigenvalue, the number of conjugate times in	 $[0,T]$ grows to infinity for $T\to \pm\infty$.
\item If $\vec{H}$ has no odd-dimensional Jordan blocks corresponding to a pure imaginary eigenvalue, there are no conjugate times.
\end{itemize}
\end{theorem}
Thus, it is sufficient to put the Hamiltonian vector field $\vec{H}$ of $\mathrm{LQ}(\y;Q)$, given by
\begin{equation}
\vec{H} \simeq \begin{pmatrix}
-\Gamma_1 & -Q \\ \Gamma_2 & \Gamma_1^*
\end{pmatrix},
\end{equation}
in its Jordan normal form, to obtain necessary and sufficient condition for the finiteness of the first conjugate time.

\begin{example}\label{ex:1}
If $\y$ is the Young diagram associated with a Riemannian geodesic, with a single column with $n=\dim M$ boxes (or, equivalently, one single level with $1$ superbox), $\Gamma_1 = 0$, $\Gamma_2 = \mathbb{I}$, and $\mathrm{LQ}(\y;\k \mathbb{I})$ is given by
\begin{equation}
H(p,x) = \frac{1}{2} \left(|p|^2 +\k |x|^2\right),
\end{equation}
which is the Hamiltonian of an harmonic oscillator (for $\k >0$), a free particle (for $\k =0$) or an harmonic repulsor (for $\k <0$). Extremal trajectories satisfy $\ddot{x} + \k x = 0$. Moreover
\begin{equation}
t_c(\y;\k\mathbb{I}) = \begin{cases} \frac{\pi}{\sqrt{k}}, & \k > 0 \\ +\infty & \k \leq 0.
\end{cases}
\end{equation}
Indeed, for $\k >0$, all extremal trajectories starting from the origin are periodic, and they return to the origin at $t = \pi/\sqrt{\k}$. On the other hand, for $\k \leq 0$, all trajectories escape at least linearly from the origin, and we cannot have conjugate times (small variations of any extremal spread at least linearly for growing time). In this case, the Hamiltonian vector field $\vec{H}$ of $\mathrm{LQ}(\y;\k\mathbb{I})$ has characteristic polynomial $P(\lambda) = (\lambda^2+\k)^n$. Therefore Theorem~\ref{t:ARS} correctly gives that the first conjugate time is finite if and only if $\k>0$.
\end{example}

\begin{example}\label{ex:3}
For any Young diagram $\y$, consider the model $\mathrm{LQ}(\y;0)$. Indeed in this case all the eigenvalues of $\vec{H}$ vanish. Thus, by Theorem~\ref{t:ARS}, one has $t_c(\y;Q) = +\infty$. 
\end{example}

In the following, when considering average comparison theorems, we will consider a particular class of models, that we discuss in the following example. 
\begin{example}\label{ex:2}
Let $\y = \ytableausetup{smalltableaux}\ytableaushort{\empty}\ldots\ytableaushort{\empty}$ be a Young diagram with a single row of length $\ell$, and $Q = \diag\{\k_1,\ldots,\k_\ell\}$. We denote these special LQ models simply $\mathrm{LQ}(\k_1,\ldots,\k_\ell)$.

In the case $\ell = 2$, Theorem \ref{t:ARS} says that $t_c(\k_1,\k_2)<+\infty$ if and only if
\begin{equation}
\begin{cases}\k_1 > 0, \\ 4 k_2 >  -\k_1^2, \end{cases} \qquad \text{or} \qquad \begin{cases}\k_1 \leq 0, \\ \k_2 >0. \end{cases}
\end{equation}
In particular, by explicit integration of the Hamiltonian flow, one can compute that, if $\k_1 >0$ and $\k_2=0$, the first conjugate time of $\mathrm{LQ}(\k_1,0)$ is $t_c(\k_1,0) = 2\pi/\sqrt{\k_1}$.
\end{example}

\subsection{General microlocal comparison theorem}

We are now ready to prove the main result on estimates for conjugate times in terms of the constant curvature models $\mathrm{LQ}(\y;Q)$.

\begin{theorem}\label{t:comparison1}
Let $\gamma(t)$ be an ample, equiregular geodesic, with Young diagram $\y$. Let $\mathfrak{R}_{\gamma(t)}: T_{\gamma(t)} M \to \R$ be directional curvature in the direction of the geodesic and $t_{c}(\g)$ the first conjugate time along $\g$. Then
\begin{itemize}
\item[(i)] if $\mathfrak{R}_{\gamma(t)} \geq Q_+$ for all $t \geq 0$, then $t_c(\g) \leq t_c(\y;Q_+),$
\item[(ii)] if $\mathfrak{R}_{\gamma(t)} \leq Q_-$ for all $t \geq 0$, then  $t_{c}(\g) \geq t_{c}(\y;Q_{-})$,
\end{itemize}
where $Q_\pm :\R^n \to \R$ are some constant quadratic forms and we understand the identification of $T_{\gamma(t)} M \simeq \R^n$ through any orthonormal basis for the scalar product $\langle\cdot|\cdot\rangle_{\gamma(t)}$.
\end{theorem}

In particular, since $t_c(\y;0) = +\infty$ (see Example~\ref{ex:3}), we have the following corollary.
\begin{corollary}\label{t:comparison2}
Let $\gamma(t)$ be an ample, equiregular geodesic, with Young diagram $\y$. Let $\mathfrak{R}_{\gamma(t)}: T_{\gamma(t)} M \to \R$ be directional curvature in the direction of the geodesic. Then, if $\mathfrak{R}_{\gamma(t)} \leq 0$ for all $t \geq 0$, there are no conjugate points along the geodesic.
\end{corollary}
In other words, the first conjugate times of $\mathrm{LQ}(\y;Q)$ gives an estimate for the first conjugate time along geodesics with directional curvature $\mathfrak{R}_{\gamma(t)}$ controlled by $Q$. 
\begin{remark}\label{r:takingout}
Notice that there is no curvature along the direction of motion, that is $\mathfrak{R}_{\gamma(t)}(\dot{\gamma}(t)) = 0$. As it is well known in Riemannian geometry, it is possible to ``take out the direction of the motion'', considering the restriction of $\mathfrak{R}_{\gamma(t)}$ to the orthogonal complement of $\dot{\gamma}(t)$, with respect to $\langle\cdot|\cdot\rangle_{\gamma(t)}$, effectively reducing the dimension by one. To simplify the discussion, we do not go into such details since there is no variation with respect to the classical Riemannian case.
\end{remark}
\begin{remark}
These microlocal theorems apply very nicely to geodesics in the Heisenberg group. In this example we have both geodesics with $\mathfrak{R}_{\gamma(t)} = 0$ (the straight lines) and geodesics with $\mathfrak{R}_{\gamma(t)} > 0$ (all the others). The former do not have conjugate times (by Theorem~\ref{t:comparison2}), while the latter do all have a finite conjugate time (by Theorem~\ref{t:comparison1}). For more details see Section~\ref{s:3d}.
\end{remark}
\begin{proof}[Proof of Theorem~\ref{t:comparison1}]
By Proposition~\ref{p:blowconj}, the study of the first conjugate time is reduced to the study of the blow-up time of the solutions of the Riccati equation.

We precise the meaning of blow-up time of a quadratic form. Let $t \mapsto V(t) :\R^n \to \R$ a continuous family of quadratic forms. For any $w \in \R^n$ let $t\mapsto w^* V(t) w$. We say that $\bar{t} \in \R \cup\{\infty\}$ is a blow-up time for $V(t)$ if there exists $w \in \R^n$ such that 
\begin{equation}
\lim_{t \to \bar{t}} w^* V(t) w \to \infty.
\end{equation}
This is equivalent to ask that one of the entries of the representative matrix of $V(t)$ grows unbounded for $t \to \bar{t}$. If $\bar{t}$ is a blow-up time for $V(t)$ and, in addition, for any $w$ such that $\displaystyle\lim_{t \to \bar{t}} w^* V(t) w = \infty$ we have $\displaystyle\lim_{t \to \bar{t}} w^* V(t) w = + \infty$ (resp. $-\infty$), we write
\begin{equation}
\lim_{t\to \bar{t}} V(t) = +\infty \qquad (\text{resp} -\infty).
\end{equation}

We compare the solution of the Cauchy problem \eqref{eq:cauchyp} for the matrix $V(t)$ for our extremal:
\begin{equation}\label{eq:ricproof}
\dot{V} = - \begin{pmatrix}
\mathbb{I} & V
\end{pmatrix}
\begin{pmatrix}
R(t) & \Gamma_1 \\
\Gamma_1^* & \Gamma_2
\end{pmatrix}
\begin{pmatrix}
\mathbb{I} \\ V
\end{pmatrix}, \qquad \displaystyle\lim_{t\to 0^+} V^{-1} = 0,
\end{equation}
and the analogous solution $V_{\y;Q}$ for any normal extremal of the model $\mathrm{LQ}(\y;Q_{\pm})$:
\begin{equation}
\dot{V}_{\y;Q_\pm} = - \begin{pmatrix}
\mathbb{I} & V_{\y;Q_\pm}
\end{pmatrix}
\begin{pmatrix}
Q_\pm & \Gamma_1 \\
\Gamma_1^* & \Gamma_2
\end{pmatrix}
\begin{pmatrix}
\mathbb{I} \\ V_{\y;Q_\pm}
\end{pmatrix}, \qquad \displaystyle\lim_{t\to 0^+} V_{\y;Q_\pm}^{-1} = 0.
\end{equation}
By Lemma~\ref{l:limit} in Appendix~\ref{a:well}, both solutions are well defined and positive definite for $t>0$ sufficiently small. By hypothesis, $R(t) \geq Q_+$ (resp. $R(t) \leq Q_-$). Therefore
\begin{equation}
-\begin{pmatrix}Q_+ & \Gamma_1 \\ \Gamma_1^* & \Gamma_2 \end{pmatrix} \geq -\begin{pmatrix}R(t) & \Gamma_1 \\ \Gamma_1^* & \Gamma_2
 \end{pmatrix} \qquad
\text{resp.}\qquad
 -\begin{pmatrix}R(t) & \Gamma_1 \\ \Gamma_1^* & \Gamma_2
 \end{pmatrix} \geq -\begin{pmatrix} Q_- & \Gamma_1 \\ \Gamma_1^* & \Gamma_2 \end{pmatrix}.  
\end{equation}
Moreover, by definition, $\lim_{t\to 0^+} V_{\y;Q_\pm}^{-1}(t) =  \lim_{t\to 0^+}  V^{-1}(t) =0$. Therefore, by Riccati comparison (Theorem~\ref{t:riccatilim} in Appendix~\ref{a:ricccati}), we obtain
\begin{equation}
V(t) \leq V_{\y;Q_+}(t), \qquad \text{resp.}\qquad V(t) \geq V_{\y;Q_-}(t),
\end{equation}
for all $t>0$ such that both solutions are defined. We need the following two lemmas.
\begin{lemma}\label{l:monotone}
For any $\y$ and $Q$, the solution $V_{\y;Q}$ is monotone non-increasing.
\end{lemma}
\begin{proof}[Proof of Lemma~\ref{l:monotone}]
It is a general fact that any solution of the symmetric Riccati differential equation with constant coefficients is monotone (see \cite[Thm. 4.1.8]{abou2003matrix}). In other words, for any solution $X(t)$ of a Cauchy problem with a Riccati equation with constant coefficients
\begin{equation}
\dot{X} +A^*X+XA+B+XQX = 0, \qquad X(t_0) = X_0,
\end{equation}
we have that $\dot{X} \geq 0$ (for $t\geq t_0$, where defined) if and only if $\dot{X}(t_0) \geq 0$ (true also with reversed and/or strict inequalities). Thus, in order to complete the proof of the lemma, it only suffices to compute the sign of $\dot{V}_{\y;Q}(\eps)$. This is easily done by exploiting the relationship with the inverse matrix $W_{\y;Q} = V_{\y;Q}^{-1}$. Observe that $\dot{W}_{\y;Q}(0) =\Gamma_2 \geq 0$. Then $W_{\y;Q}(t)$ is monotone non-decreasing. In particular $\dot{W}_{\y;Q}(\eps) \geq 0$. This, together with the fact that $W_{\y;Q}(\eps) > 0$ for $\eps$ sufficiently small (see Appendix~\ref{a:well}), implies that $\dot{V}_{\y;Q}(\eps) \leq 0$, and the lemma is proved.
\end{proof}
\begin{lemma}\label{l:monotone2}
If a solution $V(t)$ of the Riccati Cauchy problem~\eqref{eq:ricproof} blows up at time $\bar{t}$, then it blows up at $-\infty$, namely
\begin{equation}
\lim_{t\to \bar{t}} V(t) = -\infty.
\end{equation}
\end{lemma}
\begin{proof}[Proof of Lemma~\ref{l:monotone2}]
If $R(t)$ is constant, the statement is an immediate consequence of Lemma~\ref{l:monotone}. Remember that $R(t)$ is defined for all times. Then, let $q$ be the smallest eigenvalue of $R(t)$ on the interval $[0,\bar{t}]$. Indeed $R(t) \geq q \mathbb{I}$. Then, by Riccati comparison, $V(t) \leq V_{\y;q\mathbb{I}}(t)$. Since the latter is monotone non-increasing by Lemma~\ref{l:monotone}, the statement follows.
\end{proof}

Now we conclude. Case (i). In this case $R(t) \geq Q_+$. By Riccati comparison, $V(t) \leq V_{\y;Q_+}(t)$ on the interval $(0,\min\{t_c(\gamma),t_c(\y;Q_+)\})$. Assume that $t_c(\gamma) > t_c(\y;Q_+)$. Then
\begin{equation}
\lim_{t \to t_c(\y;Q_+)} V(t) \leq \lim_{t \to t_c(\y;Q_+)} V_{\y;Q_+}(t) = -\infty,
\end{equation}
which is a contradiction, then $t_c(\gamma) \leq t_c(\y;Q_+)$.

Case (ii). In this case $R(t) \leq Q_-$. By Riccati comparison, $V(t) \geq V_{\y;Q_-}$ on the interval $(0,\min\{t_c(\gamma),t_c(\y;Q_-)\})$. Assume that $t_c(\gamma) < t_c(\y;Q_-)$. Then
\begin{equation}
\lim_{t \to t_c(\gamma)} V_{\y;Q_-}(t) \leq \lim_{t \to t_c(\gamma)} V(t) = -\infty,
\end{equation}
and we get a contradiction. Thus $t_c(\gamma) \geq t_c(\y;Q_-)$. 
\end{proof}

%
\section{Average microlocal comparison theorem}\label{s:average}

In this section we prove the average version of Theorem~\ref{t:comparison1}. Recall that, with any ample, equiregular geodesic $\gamma(t)$ we associate its Young diagram $\y$. The latter is partitioned in levels, namely the sets of rows with the same length. Let $\lev_1,\ldots,\lev_\ell$ be the superboxes in some given level, of length $\ell$. The size $r$ of the level is the number of rows contained in the level (see Fig.~\ref{f:Yd3}). To the superboxes $\lev_i$ we associated the Ricci curvatures $\Ric^{\lev_i}_{\gamma(t)}$ for $i=1,\ldots,\ell$. Finally, we recall the definition anticipated in Example~\ref{ex:2}.

\begin{figure}
\centering
\begin{tikzpicture}[x=0.30mm, y=0.30mm, inner xsep=0pt, inner ysep=0pt, outer xsep=0pt, outer ysep=0pt]
\path[line width=0mm] (28.26,67.84) rectangle +(291.74,102.16);
\definecolor{L}{rgb}{0,0,0}
\definecolor{F}{rgb}{0.565,0.933,0.565}
\path[line width=0.60mm, draw=L, fill=F] (80.00,90.00) rectangle +(60.00,80.00);
\path[line width=0.60mm, draw=L] (100.00,170.00) -- (100.00,100.00);
\path[line width=0.60mm, draw=L] (100.00,170.00) -- (100.00,90.00);
\path[line width=0.60mm, draw=L] (120.00,170.00) -- (120.00,90.00);
\path[line width=0.30mm, draw=L, dash pattern=on 0.30mm off 0.50mm] (140.00,170.00) -- (160.00,170.00);
\path[line width=0.30mm, draw=L, dash pattern=on 0.30mm off 0.50mm] (140.00,90.00) -- (160.00,90.00);
\path[line width=0.60mm, draw=L, fill=F] (160.00,90.00) rectangle +(20.00,80.00);
\path[line width=0.15mm, draw=L] (80.00,150.00) -- (140.00,150.00);
\path[line width=0.15mm, draw=L] (80.00,130.00) -- (140.00,130.00);
\path[line width=0.15mm, draw=L] (80.00,110.00) -- (140.00,110.00);
\path[line width=0.15mm, draw=L] (160.00,150.00) -- (180.00,150.00);
\path[line width=0.15mm, draw=L] (160.00,130.00) -- (180.00,130.00);
\path[line width=0.15mm, draw=L] (160.00,110.00) -- (180.00,110.00);
\draw(60.00,155.00) node[anchor=base]{\fontsize{8.54}{10.24}\selectfont $\y_{a_1}$};
\draw(60.00,135.00) node[anchor=base]{\fontsize{8.54}{10.24}\selectfont $\y_{a_2}$};
\draw(60.00,95.00) node[anchor=base]{\fontsize{8.54}{10.24}\selectfont $\y_{a_r}$};
\draw(60.00,115.00) node[anchor=base]{\fontsize{8.54}{10.24}\selectfont $\vdots$};
\draw(90.00,70.00) node[anchor=base]{\fontsize{8.54}{10.24}\selectfont $\lev_1$};
\draw(110.00,70.00) node[anchor=base]{\fontsize{8.54}{10.24}\selectfont $\lev_2$};
\draw(130.00,70.00) node[anchor=base]{\fontsize{8.54}{10.24}\selectfont $\lev_3$};
\draw(170.00,70.00) node[anchor=base]{\fontsize{8.54}{10.24}\selectfont $\lev_\ell$};
\draw(150.00,70.00) node[anchor=base]{\fontsize{8.54}{10.24}\selectfont $\ldots$};
\path[line width=0.15mm, draw=L] (80.00,85.00) -- (80.00,80.00) -- (140.00,80.00) -- (140.00,85.00);
\path[line width=0.15mm, draw=L] (160.00,85.00) -- (160.00,80.00) -- (180.00,80.00) -- (180.00,85.00);
\path[line width=0.15mm, draw=L] (100.00,85.00) -- (100.00,80.00);
\path[line width=0.15mm, draw=L] (120.00,85.00) -- (120.00,80.00);
\path[line width=0.15mm, draw=L] (75.00,170.00) -- (70.00,170.00);
\path[line width=0.15mm, draw=L] (75.00,150.00) -- (70.00,150.00);
\path[line width=0.15mm, draw=L] (75.00,130.00) -- (70.00,130.00);
\path[line width=0.15mm, draw=L] (75.00,110.00) -- (70.00,110.00);
\path[line width=0.15mm, draw=L] (75.00,90.00) -- (70.00,90.00);
\path[line width=0.15mm, draw=L] (70.00,170.00) -- (70.00,130.00);
\path[line width=0.15mm, draw=L] (70.00,110.00) -- (70.00,90.00);
\path[line width=0.15mm, draw=L] (240.00,145.00) rectangle +(0.00,0.00);
\path[line width=0.30mm, draw=L, fill=F] (260.00,150.00) [rotate around={270:(260.00,150.00)}] rectangle +(40.00,40.00);
\path[line width=0.30mm, draw=L] (260.00,90.00) rectangle +(10.00,10.00);
\path[line width=0.30mm, draw=L] (260.00,100.00) rectangle +(30.00,10.00);
\path[line width=0.30mm, draw=L] (260.00,170.00) [rotate around={270:(260.00,170.00)}] rectangle +(20.00,60.00);
\path[line width=0.15mm, draw=L] (260.00,160.00) -- (320.00,160.00);
\path[line width=0.15mm, draw=L] (260.00,140.00) -- (300.00,140.00);
\path[line width=0.15mm, draw=L] (260.00,130.00) -- (300.00,130.00);
\path[line width=0.15mm, draw=L] (260.00,120.00) -- (300.00,120.00);
\path[line width=0.15mm, draw=L] (270.00,170.00) -- (270.00,100.00);
\path[line width=0.15mm, draw=L] (280.00,170.00) -- (280.00,100.00);
\path[line width=0.15mm, draw=L] (290.00,170.00) -- (290.00,110.00);
\path[line width=0.15mm, draw=L] (300.00,150.00) -- (300.00,170.00);
\path[line width=0.15mm, draw=L] (310.00,150.00) -- (310.00,170.00);
\path[line width=0.15mm, draw=L] (255.00,150.00) -- (250.00,150.00) -- (195.00,170.00) -- (190.00,170.00);
\path[line width=0.15mm, draw=L] (190.00,90.00) -- (195.00,90.00) -- (250.00,110.00) -- (255.00,110.00);
\end{tikzpicture}%
\caption{Detail of a single level of $\y$ of length $\ell$ and size $r$. It consists of the rows $\y_{a_1}, \ldots, \y_{a_r}$, each one of length $\ell$. The sets of boxes in each column are the superboxes $\lev_1,\ldots,\lev_\ell$.}\label{f:Yd3}
\end{figure}
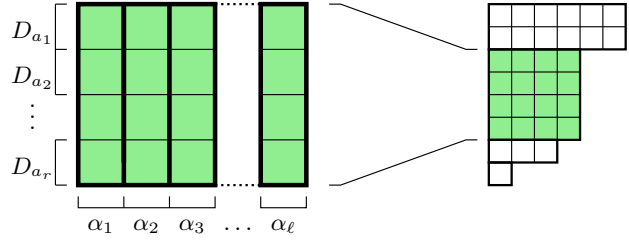

\begin{definition}
With the symbol $\mathrm{LQ}(\k_1,\ldots,\k_\ell)$ we denote the LQ model associated with the Young diagram $\y$ with a single row of length $\ell$, and with diagonal potential $Q = \mathrm{diag}(\k_1,\ldots,\k_\ell)$. With the symbol $t_c(\k_1,\ldots,\k_\ell)$ we denote the first conjugate time of $\mathrm{LQ}(\k_1,\ldots,\k_\ell)$.
\end{definition}

\begin{theorem}\label{t:comparisonaverage}
Let $\gamma(t)$ be an ample, equiregular geodesic, with Young diagram $\y$. Let $\lev_1,\ldots,\lev_\ell$ be the superboxes in some fixed level, of length $\ell$ and size $r$. Then, if
\begin{equation}
\frac{1}{r}\Ric_{\gamma(t)}^{\lev_i} \geq \k_i, \qquad \forall i=1,\ldots,\ell, \qquad \forall t \geq 0,
\end{equation}
the first conjugate time $t_c(\g)$ along the geodesic satisfies $t_c(\g) \leq t_c(\k_1,\ldots,\k_\ell)$.
\end{theorem}
The hypotheses in Theorem~\ref{t:comparisonaverage} are no longer bounds on a quadratic form, but a finite number of \emph{scalar} bounds. Observe that we have one comparison theorem for each level of the Young diagram of the given geodesic.

Consider the Young diagram of any geodesic of a Riemannian structure. It consists of a single level of length $\ell =1$, with one superbox $\lev$, of size $r = n =\dim M$ and $\Ric_{\gamma(t)}^\lev = \mathrm{Ric}^\nabla(\dot{\gamma}(t))$. This, together with the computation of $t_c(\k)$ of Example~\ref{ex:1}, recovers the following well known result.
\begin{corollary}\label{c:averriem}
Let $\gamma(t)$ be a Riemannian geodesic, such that $\mathrm{Ric}^\nabla(\dot{\gamma}(t)) \geq n \k > 0$ for all $t \geq 0$. Then the first conjugate time $t_c(\g)$ along the geodesic satisfies $t_c(\g) \leq \pi/\sqrt{\k}$.
\end{corollary}
Corollary~\ref{c:averriem} can be refined by taking out the direction of the motion, effectively reducing the dimension by $1$. A similar reduction can be performed in Theorem~\ref{t:comparisonaverage}, in the case of a ``Riemannian'' level of length $1$ and size $r$, effectively reducing the size of by one. We do not go into details, since such a reduction can be obtained exactly as in the Riemannian case (see \cite[Chapter 14]{villani} and also Remark~\ref{r:takingout}). 

We recall how the averaging procedure is carried out in Riemannian geometry. In this setting, one considers the average of the diagonal elements of $V(t)$, namely the trace, and employs the Cauchy-Schwarz inequality to obtain a scalar Riccati equation for $\trace V(t)$, where the curvature matrix is replaced by its trace, namely the Ricci curvature along the geodesic. On the other hand, in the sub-Riemannian setting, non-trivial terms containing matrices $\Gamma_1(\y)$ and $\Gamma_2(\y)$ appear in the Riccati equation. These terms, upon tracing, cannot be controlled in terms of $\trace V(t)$ alone. The failure of such a procedure in genuine sub-Riemannian manifolds is somehow expected: different directions have a different ``behaviour'', according to the structure of the Young diagram, and it makes no sense to average over all of them. The best we can do is to average among the directions corresponding to the rows of $\y$ that have the same length, namely rows in the same level. The proof of Theorem~\ref{t:comparisonaverage} is based on the following two steps.

\paragraph{(i) Splitting:} The idea is to split the Cauchy problem
\begin{equation}
\dot{V}  + \Gamma_1 V + V \Gamma_1^* + R(t) + V \Gamma_2 V = 0, \qquad \lim_{t\to 0^+}V^{-1} = 0,
\end{equation}
in several, lower-dimensional Cauchy problems for particular blocks of $V(t)$. In these equations, only some blocks of $R(t)$ appear. In particular, we obtain one Riccati equation for each row of the Young diagram $\y$, of dimension equal to the length of the row. The blow-up of a block of $V(t)$ imples a blow-up time for $V(t)$. Therefore, the presence of finite blow-up time in any one of these lower dimensional blocks implies a conjugate time for the original problem.

\paragraph{(ii) Tracing:} After the splitting step, we sum the Riccati equations corresponding to the rows with the same length, since all these equations are, in some sense, compatible (they have the same $\Gamma_1, \Gamma_2$ matrices). In the Riemannian case, this procedure leads to a single, scalar Riccati equation. In the sub-Riemannian case, we obtain one Riccati equation for each level of the Young diagram, of dimension equal to the length $\ell$ of the level. In this case the curvature matrix is replaced by a diagonal matrix, whose diagonal elements are the Ricci curvatures of the superboxes $\lev_1,\ldots,\lev_\ell$ in the given level. This leads to a finite number of scalar conditions.

\begin{proof}[Proof of Theorem~\ref{t:comparisonaverage}]
We split the blocks of the Riccati equation corresponding to the rows of the Young diagram $\y$, with $k$ rows $\y_1,\ldots,\y_k$, of length $n_1,\ldots,n_k$. Recall that the matrices $\Gamma_1(\y)$, $\Gamma_2(\y)$, defined in Eqs.~\eqref{eq:G1}-\eqref{eq:G2}, are $n\times n$ block diagonal matrices
\begin{equation}
\Gamma_i(\y) := \begin{pmatrix} 
\Gamma_i(\y_1) &  &   \\
 &  \ddots & \\
 &   & \Gamma_i(\y_k)
\end{pmatrix}, \qquad i =1,2,
\end{equation}
the $a$-th block being the $n_a\times n_a$ matrices
\begin{equation}\label{eq:Gamma}
\Gamma_1(\y_a) := \begin{pmatrix}
0 & \mathbb{I}_{n_a-1} \\
0 & 0
\end{pmatrix} , 
\qquad \Gamma_2(\y_a) := \begin{pmatrix}
1 & 0 \\
0 & 0_{n_a-1}
\end{pmatrix},
\end{equation}
where $\mathbb{I}_{m}$ is the $m \times m$ identity matrix and $0_{m}$ is the $m \times m$ zero matrix. Consider the maximal solution of the Cauchy problem
\begin{equation}
\dot{V}  + \Gamma_1 V + V \Gamma_1^* + R(t) + V \Gamma_2 V = 0, \qquad \displaystyle \lim_{t\to 0^+}V^{-1}  =0.
\end{equation}
The blow-up of a block of $V(t)$ implies a finite blow-up time for the whole matrix, hence a conjugate time. Thus, consider $V(t)$ as a block matrix. In particular, in the notation of Sec.~\ref{s:Jac}, the block $ab$, denoted $V_{ab}(t)$ for $a,b=1,\ldots,k$, is a $n_a\times n_b$ matrix with components $V_{ai,bj}(t)$, $i=1,\ldots,n_a$, $j=1,\ldots,n_b$. Let us focus on the diagonal blocks
\begin{equation}
V(t) = \begin{pmatrix}
V_{11}(t)   &   & * \\
 &   \ddots &  \\
* &   & V_{kk}(t) 
\end{pmatrix}.
\end{equation}
Consider the equation for the $a$-th block on the diagonal, which we call $V_{aa}(t)$, and is a $n_a\times n_a$ matrices with components $V_{ai,aj}(t)$, $i,j=1,\ldots,n_a$. We obtain
\begin{equation}
\dot{V}_{aa} + \Gamma_1 V_{aa} + V_{aa}\Gamma_1^* + \wt{R}_{aa}(t) + V_{aa} \Gamma_2 V_{aa} = 0,
\end{equation}
where $\Gamma_i = \Gamma_i(\y_a)$, for $i=1,2$ are the matrix in Eq.~\eqref{eq:Gamma}, i.e. the $a$-th diagonal blocks of the matrices $\Gamma_i(\y)$. Moreover
\begin{equation}
\wt{R}_{aa}(t) = R_{aa}(t) + \sum_{b\neq a} V_{ab}(t) \Gamma_2(\y_b) V_{ba}(t).
\end{equation}
The ampleness assumption implies the following limit condition for the block $V_{aa}$.
\begin{lemma}
$\displaystyle \lim_{t \to 0^+} (V_{aa})^{-1} = 0$. 
\end{lemma}
\begin{proof}
Without loss of generality, consider the first block $V_{11}$. We partition the matrix $V$ and $W = V^{-1}$ in blocks as follows
\begin{equation}
V = \begin{pmatrix}
V_{11} & V_{10} \\
V_{10}^* & V_{00}
\end{pmatrix},
\end{equation}
where the index ``0'' collects all indices different from 1. By block-wise inversion, $W_{11} = (V^{-1})_{11} = ( V_{11} - V_{10}V_{00}^{-1} V_{10}^*)^{-1}$. By Lemma~\ref{l:limit} in Appendix~\ref{a:well}, for small $t>0$, $V(t) > 0$, hence $V_{00} >0$ as well. Therefore $V_{11} - (W_{11})^{-1} = V_{10}V_{00}^{-1}V_{10}^*\geq 0$. Thus $V_{11} \geq (W_{11})^{-1} >0$ and, by positivity, $0 < (V_{11})^{-1} \leq W_{11}$ for small $t>0$. By taking the limit for $t\to 0^+$, since $W_{11} \to 0$, we obtain the statement.
\end{proof}
We proved that the block $V_{aa}(t)$ is solution of the Cauchy problem
\begin{equation}\label{eq:pretrace}
\dot{V}_{aa}+ \Gamma_1 V_{aa} + V_{aa}\Gamma_1^* +\wt{R}_{aa}(t) + V_{aa} \Gamma_2 V_{aa} = 0, \qquad \displaystyle \lim_{t\to 0^+} (V_{aa})^{-1} = 0.
\end{equation}
The crucial observation is the following (see \cite{mcpcontact} for the original argument in the contact case with symmetries). Since $\Gamma_2(\y_b) \geq 0$ and $V_{ba} = V_{ab}^*$ for all $a,b=1,\ldots,k$, we obtain 
\begin{equation}\label{eq:boundRtilde}
\wt{R}_{aa}(t) = R_{aa}(t) + \sum_{b\neq a} V_{ab}(t) \Gamma_2(\y_b) V_{ab}^*(t) \geq R_{aa}(t).
\end{equation}

We now proceed with the second step of the proof, namely tracing over the level. Consider Eq.~\eqref{eq:pretrace} for the diagonal blocks of $V(t)$, with $\wt{R}_{aa}(t) \geq R_{aa}(t)$. Now, we average over all the rows in the same level $\lev$. Let $\ell$ be the length of the level, namely $\ell = n_a$, for any row $\y_{a_1},\ldots, \y_{a_r}$ in the given level (see Fig.~\ref{f:Yd3}). Then define the $\ell \times \ell$ symmetric matrix:
\begin{equation}
V_\lev := \frac{1}{r}\sum_{a \in \lev} V_{aa},
\end{equation}
where the sum is taken on the indices $a \in \{a_1,\ldots,a_r\}$ of the rows $\y_a$ in the given level $\lev$. Once again, the blow-up of $V_\lev(t)$ implies also a blow-up for $V(t)$. A computation shows that $V_\lev$ is the solution of the following Cauchy problem
\begin{equation}
\dot{V}_\lev +\Gamma_1 V_\lev + V_\lev\Gamma_1^* + R_\lev(t) + V_\lev \Gamma_2 V_\lev = 0, \qquad \lim_{t\to 0^+} V_\lev = 0,
\end{equation}
where $\Gamma_2 = \Gamma_2(\y_{a})$ for any $a \in \lev$, and the $\ell \times \ell$ matrix $R_\lev(t)$ is defined by
\begin{equation}
\begin{split}
R_\lev(t) :=\, &\frac{1}{r}\sum_{a \in \lev} \wt{R}_{aa}(t) + \frac{1}{r}\sum_{a \in \lev} V_{aa}\Gamma_2V_{aa} - V_\lev \Gamma_2 V_\lev\\
=\, &\frac{1}{r}\sum_{a \in \lev} \wt{R}_{aa}(t) + \frac{1}{r}\left[\sum_{a \in \lev} (V_{aa}\Gamma_2)(V_{aa}\Gamma_2)^* - \frac{1}{r}\left(\sum_{a \in \lev} V_{aa}\Gamma_2\right)\left(\sum_{a \in \lev} V_{aa}\Gamma_2\right)^*\right].
\end{split}
\end{equation}
The key observation is that the term in square brackets is non-negative, as a consequence of the following lemma, whose proof is in Appendix~\ref{a:cs}.
\begin{lemma}\label{l:cs}
Let $\{X_a\}_{a=1}^r$, $\{Y_a\}_{a=1}^r$ be two sets of $\ell \times \ell$ matrices. Then
\begin{equation}\label{eq:cs}
\left(\sum_{a=1}^r X_a^* Y_a\right)\left(\sum_{b=1}^r X_b^* Y_b\right)^* \leq \left\lVert \sum_{a=1}^r Y_a^* Y_a \right\rVert \sum_{b=1}^r X_b^* X_b.
\end{equation}
Here $\lVert\cdot\rVert$ denotes the operator norm.
\end{lemma}
\begin{remark}
Lemma~\ref{l:cs} is a generalisation of the Cauchy-Schwarz inequality, in which the scalar product in $\R^r$ is replaced by a non-commutative product $\odot : \mathrm{Mat}(\ell)^r \times \mathrm{Mat}(\ell)^r \to \mathrm{Mat}(\ell)$,
such that, if $X = \{X_a\}_{a=1}^r$, $Y= \{Y_a\}_{a=1}^r$, the product $X \odot Y := \sum_{a=1}^r X^*_a Y_a$. Eq.~\eqref{eq:cs} becomes
\begin{equation}
(X\odot Y)(X\odot Y)^* \leq  \lVert Y\odot Y\rVert X\odot X.
\end{equation}
Then the l.h.s. of Eq.~\ref{eq:cs} is just the ``square of the scalar product''. For $\ell = 1$, we recover the classical Cauchy-Schwarz inequality.
\end{remark}
We apply Lemma~\ref{l:cs} to $X_a = \Gamma_2 V_{aa}$ and $Y_{a} = \Gamma_2$, for $a \in \lev = \{a_1,\ldots,a_r\}$. We obtain
\begin{equation}
\sum_{a\in \lev} (V_{aa}\Gamma_2)(V_{aa}\Gamma_2)^* - \frac{1}{r}\left(\sum_{a\in \lev} V_{aa}\Gamma_2\right)\left(\sum_{a\in \lev} V_{aa}\Gamma_2\right)^* \geq 0,
\end{equation}
which implies, together with Eq.~\eqref{eq:boundRtilde}
\begin{equation}\label{eq:almostdone}
R_\lev(t) \geq \frac{1}{r} \sum_{a \in \lev} \wt{R}_{aa}(t) \geq \frac{1}{r}\sum_{a \in \lev} R_{aa}(t).
\end{equation}
Notice that, the $ij$-th component of the sum in the r.h.s. of Eq.~\eqref{eq:almostdone} is precisely $\frac{1}{r}\sum_{a\in\lev}R_{ai,aj}(t)$, where $i,j=1,\ldots,\ell$. Thus, for any two fixed indices $i,j$ we are considering, in coordinates, the trace of the restriction $\mathfrak{R}_{\gamma(t)}: S^{\lev_i}_{\gamma(t)} \to S^{\lev_j}_{\gamma(t)}$, written in terms of any orthonormal basis for $(T_{\gamma(t)}M,\langle\cdot|\cdot\rangle_{\gamma(t)})$. The matrix $R(t)$ is normal (see Theorem~\ref{p:can}). Thus, according to \cite{lizel}, such a trace is always zero, unless $i=j$. Thus only the diagonal elements are non-vanishing and
\begin{equation}
\frac{1}{r}\sum_{a \in \lev} R_{aa}(t) = \frac{1}{r} \begin{pmatrix}
\Ric_{\gamma(t)}^{\lev_1} &   & 0 \\
 & \ddots &  \\
0 &   & \Ric_{\gamma(t)}^{\lev_\ell} \\
\end{pmatrix}.
\end{equation}
Thus, for any level $\lev$, the average over the level $V_\lev$ satisfies the $\ell \times \ell$ matrix Riccati equation
\begin{equation}
\dot{V}_\lev + \Gamma_1 V_\lev + V_\lev \Gamma_1^* + R_\lev(t) + V_\lev \Gamma_2 V_\lev = 0, \qquad \lim_{t \to 0^+} V_\lev^{-1} = 0,
\end{equation}
and, under our hypotheses, $R_\lev(t) \geq \diag\{\k_1,\ldots,\k_\ell\}$. Therefore, we proceed as in the proof of Theorem~\ref{t:comparison1}, with $\diag\{\k_1,\ldots,\k_\ell\}$ in place of $Q_+$, and we obtain the statement.
\end{proof}

\section{A sub-Riemannian Bonnet-Myers theorem}\label{s:bm}

As an application of Theorem~\ref{t:comparisonaverage}, we prove a sub-Riemannian analogue of the Bonnet-Myers theorem (see Theorem~\ref{t:bm-riem-intro} for the classical statement).


In order to globalize the previous results, we need to consider that different geodesics, even starting at the same point, may have different growth vectors. It turns out that the components of the growth vector are computed as ranks of matrices whose entries are polynomial functions of the covector $\lambda(t)$ associated with the given geodesic $\gamma(t)$. This is a direct consequence of Definition~\ref{d:flag} and the fact that the sub-Riemannian Hamiltonian is fiber-wise polynomial (actually, quadratic). It follows that, for any $x \in M$, the growth vector $\mathcal{G}_\gamma(0)$, seen as a function of the initial covector, is constant on an open Zariski subset $A_{x} \subseteq T_{x}^*M$, where it attains its (member-wise) maximal value, given by the \emph{maximal growth vector}
\begin{equation}
\mathcal{G}_x:= \{k_1(x),\ldots,k_m(x)\}, \qquad k_i(x) := \max_{\lambda\in T_x^*M} \dim\DD_\gamma^i(0).
\end{equation}
Any sub-Riemannian structure has ample geodesics starting from any given point, thus $A_x$ is not-empty for every $x \in M$ (see~\cite[Sec. 5.2]{curvature} for a proof). Moreover, all the functions $x \mapsto k_i(x)$ are bounded, lower semi-continuous with integer values. Thus, the set $\Omega\subseteq M$ of points such that $\mathcal{G}_x$ is locally constant is open and dense. 

As a consequence, the generic normal geodesic starting at $x \in \Omega$ (``generic'' means with initial covector in $A_x$) is ample and equiregular, at least when restricted to a sufficient short segment. We call $\y_x$ the Young diagram of the generic normal geodesic starting at $x$. The Young diagram $\y_x$ is ``locally constant'', in the sense that for any $x \in \Omega$ there exists an open set $U \subseteq \Omega$ such that for all $y \in U$ we have $\y_y = \y_x$.

The structure of $\Omega$ may be complicated, and a geodesics starting from $x \in \Omega$ may cross regions where $\y_x$ has different shapes. To avoid such pathological situations, we make the following assumption:
\begin{itemize}
\item[($\star$)] $\Omega = M$ and the Young diagram $\y_{x}$ is constant.
\end{itemize}
This is equivalent to the existence of a fixed Young diagram $\y$ such that, for all $x \in M$, the generic normal geodesic (i.e. with initial covector in $A_x \subseteq T_x^*M$) is ample, equiregular, with the same Young diagram $\y$. This assumption is satisfied, for instance, by any slow-growth distribution, a large class of sub-Riemannian structures including any contact, quasi-contact, fat, Engel, Goursat-Darboux distributions (see \cite[Sec. 5.5]{curvature}). Moreover, this assumption is satisfied by all left-invariant structures on Lie groups and, more generally, sub-Riemannian homogeneous spaces. 

Under the assumption $(\star)$, with a generic geodesic $\gamma(t)$ we can associate the directional curvature $\mathfrak{R}_{\gamma(t)} : T_{\gamma(t)} M \to \R$ and the corresponding Ricci curvatures $\Ric_{\gamma(t)}^\alpha$, one for each superbox $\lev$ in $\y$.

\begin{theorem}\label{t:bonnetmyers}
Let $M$ be a complete, connected sub-Riemannian manifold satisfying ($\star$). Assume that there exists a level $\alpha$ of length $\ell$ and size $r$ of the Young diagram $\y$ and constants $\k_1,\ldots,\k_\ell$ such that, for any length parametrized geodesic $\gamma(t)$
\begin{equation}
\frac{1}{r}\Ric^{\alpha_i}_{\gamma(t)} \geq \k_i, \qquad \forall i =1,\ldots,\ell, \qquad \forall t \geq 0.
\end{equation}
Then, if the polynomial
\begin{equation}
P_{\k_1,\ldots,\k_\ell}(x) := x^{2\ell} - \sum_{i=0}^{\ell -1} (-1)^{\ell-i}\k_{\ell -i} x^{2i}
\end{equation}
has at least one simple purely imaginary root, the manifold is compact, has diameter not greater than $t_c(\k_1,\ldots,\k_\ell) < +\infty$. Moreover, its fundamental group is finite.
\end{theorem}
\begin{proof}
First, we show that $\diam(M):= \sup\{d(x,y)|\,x,y \in M\} \leq t_c(\k_1,\ldots,\k_\ell)$. Let $x_0 \in M$, and let $\Sigma_{x_0} \subseteq M$ be the set of points $x$ such that there exists a unique minimizing geodesic connecting $x_0$ with $x$, strictly normal and with no conjugate points. We have the following fundamental result (see \cite{agrachevsmooth} or also \cite[Thm. 5.8]{curvature}).
\begin{theorem}
Let $x_0 \in M$. The set $\Sigma_{x_0}$ is open, dense and the sub-Riemannian squared distance $x \mapsto d^2(x_0,x)$ is smooth on $\Sigma_{x_0}$.
\end{theorem}
Indeed, the sub-Riemannian exponential map $\EXP_{x_0} : T_{x_0}^*M \to M$ is a smooth diffeomorphism between $\overline{\Sigma}_{x_0}:= \EXP_{x_0}^{-1}(\Sigma_{x_0}) \subseteq T_{x_0}^*M$ and $\Sigma_{x_0}$. Now consider all the normal geodesics connecting $x_0$ with points in $\Sigma_{x_0}$, associated with initial covectors in $\overline{\Sigma}_{x_0}$. The generic normal geodesic, with covector in $A_{x_0} \subseteq T_{x_0}^*M$ is ample and equiregular, with the same growth vector, and thus the same Young diagram $\y_{x_0} = \y$. Thus, for an open dense set $\Sigma_{x_0}' := \EXP_{x_0}(A_{x_0}) \cap \Sigma_{x_0} \subseteq M$, there exists a unique geodesic connecting $x_0$ with $x \in \Sigma_{x_0}'$, and it has Young diagram $\y$.

Now we apply Theorem~\ref{t:comparisonaverage} to all the geodesics connecting $x_0$ with points $x \in \Sigma_{x_0}'$, and we obtain that the first conjugate time $t_c$ along these geodesics satisfies $t_c \leq t_c(\k_1,\ldots,\k_\ell)$. These geodesics lose optimality after the first conjugate point and, since the geodesics are parametrised by length, we have that, for any $x_0 \in M$, $\sup\{d(x_0,x)|\,x \in \Sigma_{x_0}'\} \leq t_c(\k_1,\ldots,\k_\ell)$. By density of $\Sigma_{x_0}'$ in $M$, we obtain that $\diam(M) \leq t_c(\k_1,\ldots,\k_\ell)$. The condition on the roots of $P_{\k_1,\ldots,\k_\ell}$ implies that $t_c(\k_1,\ldots,\k_\ell) < +\infty$, by Theorem~\ref{t:ARS}.

By completeness of $M$, closed sub-Riemannian balls are compact, hence $M$ is compact. For the result about the fundamental group, the argument is the classical one. First, we consider the universal cover $\wt{M}$ of $M$. We define a sub-Riemannian structure on $\wt{M}$ uniquely by lifting the sub-Riemannian metric on the evenly covered neighbourhoods. All the local assumptions of our theorem remain true also on $\wt{M}$. The completeness remains true as well. The bounds on Ricci curvature still holds since it is only a local concept and the covering map is a local isometry. Hence we apply Theorem~\ref{t:comparisonaverage} to the covering, and also $\wt{M}$ is compact. Then any point $q \in M$ has a finite number of preimages in $\wt{M}$, and so $\pi_1(M)$ is finite since $\pi_1(\wt{M})$ is trivial.
\end{proof}
\begin{remark}
In the Riemannian case, $P_{\k_1}(x) = x^2 + \k_1$. Then we recover the classical Bonnet-Myers theorem since, by Example~\ref{ex:1}, $t_c(\k_1) = \pi/\sqrt{\k_1}$.
\end{remark}
%
%
\section{Applications to left-invariant structures on 3D unimodular Lie groups}\label{s:3d}

Consider a contact left-invariant sub-Riemannian structure on a 3D manifold. Any non-trivial geodesic is ample, equiregular and has the same Young diagram, with two boxes on the first row, and one in the second row (see Example~\ref{ex:contact}). The subspace associated with the box in the second row corresponds to the direction of the motion, i.e., the tangent vector to the geodesic. Since the curvature always vanishes in this direction (see Remark~\ref{r:takingout}), we can restrict to a single-level Young diagram $\lev$ of length $\ell=2$ and size $r=1$. We denote by $\lev_{1},\lev_{2}$ the two boxes of this level. The comparison LQ model is the one discussed in Example~\ref{ex:2}. Then, Theorem~\ref{t:comparisonaverage} rewrites as follows.
\begin{theorem}\label{t:bonnetmyers2}
Let $\gamma(t)$ be a length parametrized geodesic of a contact left-invariant sub-Rie\-mannian structure on a 3D manifold. Assume that
\begin{equation}
\Ric_{\g(t)}^{\lev_{i}} \geq \k_{i}, \qquad i=1,2, \qquad \forall t \geq 0,
\end{equation}
for some $\k_1$, $\k_2$ such that
\begin{equation}\label{eq:3dcond}
\begin{cases}
\k_{1} > 0, \\ 4\k_{2} > -\k_{1}^2,
\end{cases}
\qquad \text{or} \qquad
\begin{cases}
\k_{1} \leq 0, \\ \k_{2} >0.
\end{cases}
\end{equation}
Then $t_{c}(\g)\leq t_{c}(\k_{1},\k_{2})<+\infty$. If the hypotheses are satisfied for every length parametrised geodesic, then
the manifold is compact, with diameter not greater than $t_c(\k_1,\k_2)$. Moreover, its fundamental group is finite.
\end{theorem}
In the statement of Theorem~\ref{t:bonnetmyers2} we allow also for negative Ricci curvatures. Indeed, in general, $\Ric_{\g(t)}^{\lev_2}$ is not sign-definite along the geodesic.

\subsection{Invariants of a 3D contact structure}

For \emph{invariant} of a sub-Riemannian structure we mean any scalar function that is preserved by isometries. In this section we introduce the invariants $\chi$, $\kappa$ of 3D contact sub-Riemannian structures, not necessarily left-invariant. For left-invariant structures, $\chi$ and $\kappa$ are constant and we write the expression for $\Ric_{\g(t)}^{\lev_1}$ and $\Ric_{\g(t)}^{\lev_2}$ in terms of these quantities. The presentation follows closely the one contained in \cite{miosr3d}, where the interested reader can find more details.

Recall that a three-dimensional sub-Riemannian structure is contact if $\distr= \ker \omega$, where $d\omega |_{\distr_{x}}$ is non degenerate, for every $x\in M$. In what follows we normalize the contact structure by requiring that $d\omega |_{\distr_{x}}$ agrees with the volume induced by the inner product on $\distr$.  The \emph{Reeb vector field} associated with the contact structure is the unique vector field $X_{0}$ such that $\omega(X_{0})=1$ and $d\omega(X_{0},\cdot)=0$. Notice that $X_{0}$ depends only on the sub-Riemannian structure. For every orthonormal frame $X_{1},X_{2}$ on the distribution, we have
\begin{equation}\label{eq:algebracampi0}
\begin{split} 
[X_{1},X_{0}]&=c_{01}^1 X_{1}+c_{01}^2 X_{2},\\
[X_{2},X_{0}]&=c_{02}^1 X_{1}+c_{02}^2 X_{2}, \\
[X_{2},X_{1}]&=c_{12}^1 X_{1}+c_{12}^2 X_{2}+X_{0},
\end{split}
\end{equation}
where $c_{ij}^{k}\in C^{\infty}(M)$. The sub-Riemannian Hamiltonian is
\begin{equation} \label{eq:ham3d}
H=\frac{1}{2}(h_1^2+h_2^2),
\end{equation}
where $h_i(\lam)=\langle \lam,X_{i}(q) \rangle$ are the linear-on-fibers functions on $T^{*}M$
associated with the vector fields $X_{i}$, for $i=0,1,2$. Length parametrized geodesics are projections of solutions of the Hamiltonian system associated with $H$ on $T^{*}M$ that are contained in the level set $H=1/2$.

The Poisson bracket $\{H,h_0\}$ is an invariant of the sub-Riemannian structure and, by definition, it vanishes everywhere if and only if the flow of the Reeb vector field $e^{tX_{0}}$ is a one-parameter family of sub-Riemannian isometries. A standard computation gives 
\begin{equation} \label{eq:parentesihh0}
\{H,h_0\}=c_{01}^1h_1^2+(c_{01}^2+c_{02}^1)h_1h_2+c_{02}^2h_2^2.
\end{equation}
For every $x\in M$, the restriction of $\{H,h_0\}$ to $T^{*}_{x}M$, that we denote by $\{H,h_0\}_{x}$, is a quadratic form on the dual of the distribution $\distr^{*}_{x}$, hence it can be interpreted as a symmetric operator on the distribution $\distr_{x}$ itself. In particular its determinant and its trace are well defined. Moreover one can show that $\trace\{H,h_0\}_x=c_{01}^1+c_{02}^2=0$, for every $x\in M$.
 The first invariant $\chi$ is defined as the positive eigenvalue of this operator, namely
\begin{equation} \label{eq:defchi}
 \chi(x):=\sqrt{-\det\{H,h_0\}_x}\geq0.
\end{equation}
The second invariant $\kappa$
can be defined  via the structure constants \eqref{eq:algebracampi0} as follows:
\begin{equation} \label{eq:defkappa}
\kappa(x):=X_{2}(c_{12}^1)-X_{1}(c_{12}^2)-(c_{12}^1)^2-(c_{12}^2)^2+
\frac{c_{01}^2-c_{02}^1}{2}.
\end{equation}  
One can prove that the expression \eqref{eq:defkappa} is invariant by rotation of the orthonormal frame.
\begin{remark}
The quantities $\chi$, $\kappa$ were first introduced in~\cite{agrexp} as differential invariants appearing in the asymptotic expansion of the cut and conjugate locus of the sub-Riemannian exponential map near to the base point.
\end{remark}

\subsection{Left-invariant structures}
For left-invariant structures, the functions $c_{ij}^{k}$ and the invariants $\chi$, $\kappa$, are constant on $M$ can be used to classify the left-invariant structures on three-dimensional Lie groups. In particular, when $\chi=0$, the unique left-invariant structures (up to local isometries) are the Heisenberg group $\mathbb{H}$ and the Lie groups $\mathrm{SU}(2)$ and $\mathrm{SL}(2)$, with metric given by the Killing form, corresponding to the choice of $\kappa=0,1,-1,$ respectively. When $\chi>0$, for each choice of $(\chi,\kappa)$ there exists exactly one unimodular Lie group with these values (see \cite[Thm. 1]{miosr3d}).

\subsubsection{Case \texorpdfstring{$\chi=0$}{X=0}}\label{s:chi0}
When $\chi=0$, the flow of the Reeb vector field is a one-parameter family of sub-Riemannian isometries. In particular from the computations contained in \cite[Thm. 6.2]{AAPL} one gets that the directional curvature $\mathfrak{R}_{\gamma(t)}$ is diagonal with entries  $\Ric_{\g(t)}^{\lev_{1}}$ and $\Ric_{\g(t)}^{\lev_{2}}$, where
\begin{equation}
\Ric_{\g(t)}^{\lev_{1}}=h_{0}^{2}(t)+\kappa(h_{1}^{2}(t)+h_{2}^{2}(t)),\qquad \Ric_{\g(t)}^{\lev_{2}}=0.
\end{equation}
We stress that  $\lam(t)=(h_{0}(t),h_{1}(t),h_{2}(t))$ is the solution of the Hamiltonian system associated with $H$ and $\lam=(h_{1}(0),h_{2}(0),h_{0}(0))$ is the initial covector associated with the geodesic $\g(t)$.

For any length-parametrized geodesic $H(\lam(t))=1/2$, namely $h_{1}^{2}(t)+h_{2}^{2}(t)=1$. Moreover $h_{0}(t)= h_0$ is a constant of the motion. Thus
\begin{equation}
\Ric_{\g(t)}^{\lev_{1}}=h_{0}^{2}+\kappa,\qquad \Ric_{\g(t)}^{\lev_{2}}=0.
\end{equation}
Notice that $\Ric_{\g(t)}^{\lev_{1}}$ and $\Ric_{\g(t)}^{\lev_{2}}$ are constant in $t$ and $\mathfrak{R}_{\gamma(t)}$ is diagonal, so for all these cases we can apply  Theorem \ref{t:comparison1}, computing the exact value of the first conjugate time. In particular, this recovers the following well known results obtained in \cite{gersh,boscainrossi}.

\begin{itemize}
\item $\mathbb{H}$. In this case $\kappa=0$. If $h_{0}=0$ we have $\Ric_{\g(t)}^{\lev_{1}}=\Ric_{\g(t)}^{\lev_{2}}=0$ and the geodesic has no conjugate point. If $h_{0}\neq 0$ then $t_{c}=2\pi/|h_{0}|$.

\item $\mathrm{SU}(2)$. In this case $\kappa=1$. We have $\Ric_{\g(t)}^{\lev_{1}}=h_{0}^{2}+1$, $\Ric_{\g(t)}^{\lev_{2}}=0$ and every geodesic has  conjugate time $t_{c}=2\pi/\sqrt{h_{0}^{2}+1}$.

\item $\mathrm{SL}(2)$. In this case $\kappa=-1$. We have $\Ric_{\g(t)}^{\lev_{1}}=h_{0}^{2}-1$, $\Ric_{\g(t)}^{\lev_{2}}=0$ and we have two cases. If $h_{0}\leq 1$ then $t_{c}=+\infty$. If $h_{0}>1$ every geodesic has  conjugate time $t_{c}=2\pi/\sqrt{h_{0}^{2}-1}$.
\end{itemize}
Let us mention that, for $\mathrm{SU}(2)$, the first condition of~\eqref{eq:3dcond} holds for any geodesic. Hence, thanks to Theorem \ref{t:bonnetmyers2}, we recover its compactness and the exact estimate on its diameter, equal to $2\pi$. 
\subsubsection{Case \texorpdfstring{$\chi>0$}{X > 0}}
In this section we prove our result on 3D unimodular Lie groups with $\chi>0$.
Let us recall that under these assumptions, there exists a special orthonormal frame for the 
sub-Riemannian structure. In terms of the latter we provide the explicit expression of a constant of the motion.

\begin{proposition} \label{p:framenon0} 
Let $M$ be a 3D unimodular Lie group, endowed with a contact left-invariant structure, with $\chi >0$. Then there exists a left-invariant orthonormal frame $X_{1},X_{2}$ on the distribution such that 
\begin{equation} \label{eq:parconchi} \{H,h_0\}=2\chi h_1h_2. \end{equation} 
Moreover the Lie algebra defined by the frame
$X_{0},X_{1},X_{2}$ satisfies 
\begin{equation}\label{eq:algchinon00}
\begin{split}
    [X_{1},X_{0}]&=  (\chi+\kappa)X_{2} , \\
    [X_{2},X_{0}]&=(\chi-\kappa)X_{1}, \\
    [X_{2},X_{1}]&=X_{0}.
\end{split}
\end{equation}
The function $E:T^{*}M\to \R$, defined by
\begin{equation}
E=\frac{h_{0}^{2}}{2\chi}+h_{2}^{2},
\end{equation}
is a constant of the motion, i.e., $\{H,E\}=0$.  Finally the curvatures $\Ric_{\g}^{\lev_{1}}$ and $\Ric_{\g}^{\lev_{2}}$ satisfy:
\begin{gather}
\Ric_{\g}^{\lev_{1}}=h_{0}^{2}+3 \chi(h_{1}^{2}-h_{2}^{2})+\kappa(h_{1}^{2}+h_{2}^{2}),\label{eq:R11}\\
\Ric_{\g}^{\lev_{2}}=6\chi (h_{1}^{2}-h_{2}^{2})h_{0}^{2}-2\chi(\chi+\kappa)h_{1}^{4}-12\chi^{2}h_{1}^{2}h_{2}^{2}-2\chi(\chi-\kappa)h_{2}^{4}. \label{eq:R22}
\end{gather}
\end{proposition}
In Eqs. \eqref{eq:R11} and \eqref{eq:R22} we suppressed the explicit dependence on $t$. 
\begin{proof} From \cite[Prop.\ 13]{miosr3d} it follows that
there exists a unique (up to a sign) canonical frame
$X_{0},X_{1},X_{2}$ such that 
\begin{equation} \label{eq:algebracampi1}
\begin{split}
    [X_{1},X_{0}]&=  c_{01}^2X_{2} , \\
    [X_{2},X_{0}]&=c_{02}^1X_{1}, \\
[X_{2},X_{1}]&=c_{12}^1 X_{1}+c_{12}^2 X_{2}+X_{0}.
\end{split}
\end{equation}
In particular, if the Lie group is unimodular, then the left and the right Haar measures coincide. This implies $c_{12}^{1}=c_{12}^{2}=0$ (cf. proof of \cite[Thm. 1]{miosr3d}).
Then, from \eqref{eq:defchi} and \eqref{eq:defkappa}, it  follows that $\chi=(c_{01}^2+c_{02}^1)/2$, and  $\kappa=(c_{01}^2-c_{02}^1)/2$, which imply  \eqref{eq:algchinon00}.

Let us show that, if \eqref{eq:algchinon00} holds, then $\{H,E\}=0$. Using that $\{H,h_0\}=2\chi h_1h_2$ and $\{H,h_{2}\}=\{h_{1},h_{2}\}h_{1}=-h_{0}h_{1}$ one gets
\begin{equation}
\{H,E\}=\frac{1}{\chi}\{H,h_{0}\}h_{0}+2\{H,h_{2}\}h_{2} =2 h_1h_2h_{0} -2h_{1}h_{2}h_{0}=0.
\end{equation}
Finally, Eqs.~\eqref{eq:R11} and~\eqref{eq:R22} are simply formulae form \cite[Thm. 6.2]{AAPL} specified for left-invariant structures and rewritten in terms of $\chi,\kappa$ in the frame  introduced above (notice that the constants $c_{ij}^{k}$ appearing here are the opposite of those  used in \cite{AAPL}).
\end{proof}

Since $E$ is a constant of the motion, for any length parametrized geodesic $\g(t)$ we denote by $E(\g)$ the (constant) value of $E(\lam(t))$, where $\lam(t)$ is the solution of the Hamiltonian system associated with $H$ such that $\g(t)=\pi(\lam(t))$.

\begin{theorem} \label{t:Egrande}
Let $M$ be a 3D unimodular Lie group, endowed with a contact left-invariant structure, with $\chi >0$ and $\kappa\in\R$. Then there exists $\barE=\barE(\chi,\kappa)$ such that every length parametrized geodesic $\gamma$ with $E(\g)\geq \barE$ has a finite conjugate time.  
\end{theorem}

\begin{proof}
We prove that the assumptions of Theorem \ref{t:bonnetmyers2} are satisfied for every geodesic when $E$ is large enough.
Since $E$ is a constant of  the motion and $H=1/2$ we have
\begin{equation}\label{eq:HE}
h_{2}^{2}=E- \frac{h_{0}^{2}}{2\chi},\qquad h_{1}^{2}=1- E+ \frac{h_{0}^{2}}{2\chi}.
\end{equation}
Plugging Eq.~\eqref{eq:HE} into Eqs.~\eqref{eq:R11} and \eqref{eq:R22},  $\Ric_{\g(t)}^{\lev_{1}}$ and $\Ric_{\g(t)}^{\lev_{2}}$  are rewritten as follows
\begin{gather}
\Ric_{\g(t)}^{\lev_{1}}=4 h_{0}^{2}-3 \chi(2E-1)+\kappa, \label{eq:R11s}\\
\Ric_{\g(t)}^{\lev_{2}}=8 h_{0}^{4}-[2\kappa+10\chi(2E-1)]h_{0}^{2} +[2\chi\kappa(2E-1)+\chi^{2}(8E^{2}-8E-2)].\label{eq:R220}
\end{gather}
Since $h_{1}^{2}+h_{2}^{2}=1$ one has $|h_{2}|\leq 1$, from \eqref{eq:HE} one has the following bound for $h_{0}$ along the curve 
\begin{equation}\label{eq:boundh0}
2\chi (E-1)\leq h_{0}^{2}(t)\leq 2\chi E.
\end{equation}
Then we have easily a lower bound for $\Ric_{\g(t)}^{\lev_{1}}$
\begin{equation}
\begin{split}
\Ric_{\g(t)}^{\lev_{1}}&\geq 8\chi (E-1)-3 \chi(2E-1)+\kappa\\
&\geq 2\chi E-5\chi +\kappa=:\k_{1}
\end{split}
\end{equation}
Since we want to prove the result for $E$ large enough, we assume that
\begin{equation}
E\geq \frac{1}{2}\left(5-\frac{\kappa}{\chi}\right),
\end{equation}
so that $\k_{1}>0$, and the coefficient of $h_{0}^{2}$ in \eqref{eq:R220} is negative. Then using \eqref{eq:boundh0}  one estimates 
\begin{align}
\Ric_{\g(t)}^{\lev_{2}}
&\geq 2\chi^{2}(15-26E)-2\chi\kappa=:\k_{2}
\end{align}
In order to show that the first condition of Eq.~\eqref{eq:3dcond} of Theorem~\ref{t:bonnetmyers2} is satisfied we also compute
\begin{equation} \label{eq:c1c2}
4\k_{2}+\k_{1}^{2}=4\chi^{2}E^{2}+a(\chi,\kappa)E+b(\chi,\kappa),
\end{equation}
where $a$ and $b$ are the following quadratic functions
\begin{equation}
a(\chi,\kappa)=4\chi \kappa -228 \chi^{2},\qquad 
b(\chi,\kappa)=145\chi^{2}-18\chi \kappa +\kappa^{2}.
\end{equation}
Since the coefficient of $E^{2}$ in Eq.~\eqref{eq:c1c2} is positive, there exists $\barE=\barE(\chi,\kappa)$, the largest positive root of Eq.~\eqref{eq:c1c2}, such that $4\k_{2}+\k_{1}^{2}>0$  for all $E>\barE$, which ends the proof.
\end{proof}
\begin{remark}
The roots of Eq.~\eqref{eq:c1c2}, and in particular $\barE(\chi,\kappa)$, depend only on the ratio $\kappa/\chi$. This means that this number is invariant by rescaling of the sub-Riemannian structure. This could seem strange at a first glance but is a consequence of the fact that we consider only length parametrized geodesics. We also stress that, in general, the value $\barE(\chi,\kappa)$ given by this computation is not sharp.
\end{remark}

\appendix
\section{Comparison theorems for the matrix Riccati equation} \label{a:ricccati}

The general, non-autonomous, symmetric matrix Riccati equation can be written as follows:
\begin{equation}
\dot{X} = \mathrm{R}(X;t):=M(t)_{11} +  X M(t)_{12} + M(t)_{12}^* X + X M(t)_{22} X = \begin{pmatrix}
\mathbb{I} & X
\end{pmatrix}
M(t)\begin{pmatrix}
\mathbb{I} \\ X 
\end{pmatrix},
\end{equation}
where $M(t)$ is a smooth family of $2n\times 2n$ symmetric matrices. We always assume a symmetric initial datum, then the solution must be symmetric as well on the maximal interval of definition. All the comparison results are based upon the following theorems.
\begin{theorem}[Riccati comparison theorem 1]\label{t:riccati}
Let $M_1(t)$, $M_2(t)$ be two smooth families of $2n\times 2n$ symmetric matrices. Let $X_i(t)$ be smooth solution of the Riccati equation
\begin{equation}
\dot{X}_i = \mathrm{R}_i(X_i;t), \qquad i =1,2,
\end{equation}
on a common interval $I \subseteq \R$. Let $t_0 \in I$ and \emph{(i)} $M_1(t) \geq M_2(t)$ for all $t \in I$, \emph{(ii)} $X_{1}(t_0) \geq X_{2}(t_0)$. Then for any $t \in [t_0,+\infty) \cap I$, we have $X_1(t) \geq X_2(t)$.
\end{theorem}
\begin{proof}
The proof is a simplified version of \cite[Thm. 4.1.4]{abou2003matrix}. Let $U:= X_1 - X_2$. Notice that $U$ is symmetric on the interval $I$ where both solutions are defined. A computation shows that
\begin{equation}
\dot{U} = \theta(t) U + U \theta(t)^* + \begin{pmatrix}
\mathbb{I}& X_1
\end{pmatrix} (M_1 - M_2)\begin{pmatrix}
\mathbb{I} \\ X_1
\end{pmatrix},
\end{equation}
where
\begin{equation}
\theta(t) = M_2(t)_{12}^* + \frac{1}{2}X_1(t) M_2(t)_{22} + \frac{1}{2}X_2(t) M_2(t)_{22}.
\end{equation}
Taking in account that $M_1(t) - M_2(t) \geq 0$, the matrix $U$ satisfies
\begin{equation}
\dot{U} \geq \theta(t)U + U\theta(t)^*.
\end{equation}
Indeed $U(t_0) \geq 0$. Then, the statement follows from the next lemma (see \cite[Thm. 4.1.2]{abou2003matrix}).
\begin{lemma}
Let $U$ be a symmetric solution of the Lyapunov differential inequality 
\begin{equation}
\dot{U} \geq \theta(t) U + U \theta(t)^*, \qquad t \in I\subseteq \R,
\end{equation}
where $\theta(t)$ is smooth. Then $U(t_0) \geq 0$ implies $U(t) \geq 0$ for all $t \in I \cap [t_0,+\infty)$.\qedhere
\end{lemma}
\end{proof}

The assumptions of Theorem~\ref{t:riccati}  involve comparison on coefficients of Riccati equations and on initial data. It can be generalised also for limit initial data as follows.
\begin{theorem}[Riccati comparison theorem 2]\label{t:riccatilim}
Let $M_1(t)$, $M_2(t)$ be two smooth families of $2n\times 2n$ symmetric matrices. Let $X_i(t)$ be smooth solutions of the Riccati equation
\begin{equation}
\dot{X}_i = \mathrm{R}_i(X_i;t), \qquad i =1,2,
\end{equation}
on a common interval $I \subseteq \R$. Let $t_0 \in \bar{I}$. Assume that \emph{(i)} $M_1(t) \geq M_2(t)$ for all $t \in \bar{I}$, \emph{(ii)} $X_i(t) >0$ for $t>t_0$ sufficiently small, \emph{(iii)} there exist $Y_i(t_0):=\lim_{t\to t_0+} X_i^{-1}(t)$ and \emph{(iv)} $Y_1(t_0) \leq Y_2(t_0)$. Then, for any $t \in (t_0,+\infty) \cap I$, we have $X_1(t) \geq X_2(t)$.
\end{theorem}
\begin{proof}
Let $Y_i(t):= X_i(t)^{-1}$, defined on some interval $(t_0,\eps) \subseteq I$. They satisfy
\begin{equation}
\dot{Y}_i = \begin{pmatrix} \mathbb{I} & Y_i
\end{pmatrix}
N_i(t)\begin{pmatrix}
\mathbb{I} \\ Y_i
\end{pmatrix}, \qquad N_i(t):= - \begin{pmatrix}
0 & \mathbb{I} \\ \mathbb{I} & 0
\end{pmatrix} M_i(t) \begin{pmatrix}
0 & \mathbb{I} \\ \mathbb{I} & 0
\end{pmatrix}, \qquad i=1,2.
\end{equation}
Indeed $Y_i(t)$ can be prolonged on $[t_0,\eps]$ for $\eps$ sufficiently small by (iii). Moreover $N_2(t) \geq N_1(t)$ by (i) and $Y_2(t_0) \geq Y_1(t_0)$ by (iv). The point $t_0$ belongs to the interval of definition of $Y_i$ then, by Theorem~\ref{t:riccati}, $Y_2(\eps) \geq Y_1(\eps)$. By (ii), this implies that $X_1(\eps) \geq X_1(\eps)$. Then we can apply again Theorem~\ref{t:riccati} to $X_1$ and $X_2$, with $t_0 = \eps$, and we obtain that $X_1(t) \geq X_2(t)$ for all $t \in [\eps,+\infty) \cap I$. Since $\eps$ can be chosen arbitrarily close to $t_0$, we obtain the statement.
\end{proof}

\subsection{Well posedness of limit Cauchy problem}\label{a:well}
The following lemma justifies the savage use of the Cauchy problem with limit initial condition. Let $A,B$, be $n\times n$ and $n\times k$ matrices, respectively, satisfying the controllability condition
\begin{equation}
\spn \{ B,AB,\ldots,A^m B\} = \R^n,
\end{equation}
for some $m \geq 0$. Thus, since the column space of $B$ is equal to the column space of $BB^*$, we have that, if we put $\Gamma_1:=A^*$ and $\Gamma_2:= BB^* \geq 0$,
\begin{equation}
\spn \{ \Gamma_2,\Gamma_1\Gamma_2,\ldots,\Gamma_1^m\Gamma_2\} = \R^n.
\end{equation}
This condition is indeed satisfied for the matrices $\Gamma_1,\Gamma_2$ introduced in Sec.~\ref{s:Jac}.

\begin{lemma}\label{l:limit}
For any smooth {\blu and symmetric} $R(t)$, the Cauchy problem with limit initial condition
\begin{equation}\label{eq:riccatiappendix}
\dot{V} = -\Gamma_1 V - V\Gamma_1^* - R(t) - V\Gamma_2 V, \qquad \displaystyle \lim_{t\to 0^+} V^{-1} = 0,
\end{equation}
is well posed, in the sense that there exists a solution of the Riccati equation, invertible for small $t>0$ such that $ \lim_{t\to 0^+} V^{-1} = 0$. The solution is {\blu symmetric and} unique on some maximal interval of definition $I \subseteq (0,+\infty)$.
In addition, $V(t) >0$ for small $t>0$.
\end{lemma}
\begin{proof}
We first prove uniqueness. If two solutions $V_1$, $V_2$ exist, their inverses $W_1$ and $W_2$ (defined for $t>0$ sufficiently small) can be extended to smooth matrices on $[0,\eps)$, by setting $W_1(0) = W_2(0) = 0$. Moreover, they both satisfy the following Cauchy problem:
\begin{equation}
\dot{W} = \Gamma_1^* W + W \Gamma_1 + \Gamma_2 + W R(t) W, \qquad W(0) = 0.
\end{equation}
By uniqueness of the standard Cauchy problem, $W_1(\eps) = W_2(\eps)$. Therefore also $V_1^{-1}(\eps) = V_{2}^{-1}(\eps)$, and uniqueness follows. The choice of $\eps >0$ for setting the Cauchy datum is irrelevant, since different choices bring to the same solution. Finally, any solution can be extended uniquely to a maximal solution, defined on some interval $I \subseteq (0,+\infty)$.  {\blu Notice $V(t)$ and $V(t)^*$ are both solution of \eqref{eq:riccatiappendix}, in particular $V(t) = V(t)^*$.}

Now, we prove the existence. Consider the Cauchy problem
\begin{equation}
\dot{W} = \Gamma_1^* W + W \Gamma_1 + \Gamma_2 + W R(t) W, \qquad W(0) = 0.
\end{equation}
Its solution is well defined for $t \in [0,\eps)$. We will soon prove that, for $t \in (0,\eps)$, such a solution is positive. Thus $V(t):=W(t)^{-1}$, defined for $t \in (0,\eps)$, is a solution of the original Cauchy problem with limit initial datum, by construction.

We are left to prove that, for $t >0$ small enough, $W(t)>0$. Since $R(t)$ is smooth, for $t \in [0,\eps)$ we can find $\k$ such that $R(t) \geq \k \mathbb{I}$. By comparison Theorem~\ref{t:comparison1}, we have that our solution is bounded below by the solution with $R(t) = \k \mathbb{I}$. We write $W(t) \geq W_\k(t) \geq 0$ for $t \in [0,\eps)$ (the last inequality follows again from Theorem~\ref{t:comparison1}, by considering the trivial solution of the Cauchy problem obtained by setting $\Gamma_2 = 0$).

Assume that, for some small $t>0$ and $x\neq 0$, we have $W(t)x=0$. This implies $W_\k(t) x = 0$. Being a solution of a Riccati equation with constant coefficients, $W_\k(t)$ is monotone non-decreasing (indeed $\dot{W}(0) = \Gamma_2 \geq 0$, and the same holds true for $t \in [0,\eps)$ by Lemma~\ref{l:monotone}). Therefore $W_\k(t)x = 0$ identically. Therefore all the derivatives, computed at $t=0$, vanish identically. This imples, after careful examination of the higher derivatives, that
\begin{equation}
\Gamma_2 x = \Gamma_2 \Gamma_1 x = \ldots \Gamma_2\Gamma_1^m x = \ldots = 0,
\end{equation}
that leads to $x=0$. This contradicts the assumption, hence $W(t) >0$ for $t$ sufficiently small.
\end{proof}

\section{Proof of Lemma~\ref{l:cs}} \label{a:cs}
\begin{lemma*}
Let $\{X_a\}_{a=1}^r$, $\{Y_a\}_{a=1}^r$ be sets of $\ell \times \ell$ matrices. Then
\begin{equation}
\left(\sum_{a=1}^r X_a^* Y_a\right)\left(\sum_{b=1}^r X_b^* Y_b\right)^* \leq \left\lVert \sum_{a=1}^r Y_a^* Y_a \right\rVert \sum_{b=1}^r X_b^* X_b.
\end{equation}
\end{lemma*}
\begin{proof}
Let $v \in \R^\ell$. Then
\begin{equation}
v^*\left(\sum_{a=1}^r X_a^* Y_a\right)\left(\sum_{b=1}^r X_b^* Y_b\right)^* v = \left\lVert \sum_{a=1}^r Y_a^* X_a v \right\rVert^2.
\end{equation}
Notice the change in position of the transpose. Now, let $u \in \R^\ell$, such that $\lVert u \rVert = 1$, and
\begin{equation}
u^* \left(\sum_{a=1}^r Y_a^* X_a v\right) = \left\lVert \sum_{a=1}^r Y_a^* X_a v \right\rVert.
\end{equation}
Then, by the Cauchy-Schwarz inequality, we obtain
\begin{equation}\label{eq:dis1}
\left\lVert \sum_{a=1}^r Y_a^* X_a v \right\rVert^2 =  \left\lvert\sum_{a=1}^r (Y_a u)^* (X_a v) \right\rvert^2 \leq  \sum_{a=1}^r \lVert Y_a u \rVert^2 \sum_{a=1}^r \lVert X_a v \rVert^2.
\end{equation}
Now, observe that 
\begin{equation}\label{eq:dis2}
\sum_{a=1}^r\lVert Y_a u\rVert^2 = u^* \sum_{a=1}^r Y_a^* Y_au \leq \left\lVert \sum_{a=1}^r Y^*_a Y_a  \right\rVert.
\end{equation}
Then, plugging Eq.~\eqref{eq:dis2} in Eq.~\eqref{eq:dis1}, we obtain
\begin{equation}
\left\lVert \sum_{a=1}^r Y_a^* X_a v \right\rVert^2  \leq \left\lVert \sum_{a=1}^r Y^*_a Y_a  \right\rVert \sum_{b=1}^r \lVert X_b v \rVert^2,
\end{equation}
which implies the statement.
\end{proof}

\section{Proof of Proposition~\ref{p:riemcan} and Lemma~\ref{l:Riemanncurv}}\label{a:riemcan}

In order to prove Proposition~\ref{p:riemcan} and Lemma~\ref{l:Riemanncurv} we define a local frame on $T^*M$, associated with the choice of a local frame $X_1,\dots,X_n$ on $M$. For $i = 1,\dots,n$ let $h_i :T^*M \to \mathbb{R}$ be the linear-on-fibres function defined by $ h_i(\lambda) :=  \langle \lambda, X_i\rangle$. The action of derivations on $T^*M$ is completely determined by the action on affine functions, namely functions $a \in C^\infty(T^*M)$ such that $a(\lambda) = \langle \lambda, Y \rangle + \pi^* g$ for some $Y\in \VecM$, $g\in C^\infty(M)$. Then, we define the \emph{coordinate lift of a field} $X \in \VecM$ as the field $\wt{X} \in  \VecTM$ such that $\wt{X}(h_i) = 0$ for $i=1,\dots,n$ and $\wt{X}(\pi^*g) = X(g)$. This, together with Leibniz rule, characterize the action of $\wt{X}$ on affine functions, and then completely define $\wt{X}$. Indeed, by definition, $\pi_* \wt{X} = X$. On the other hand, we define the (vertical) fields $\partial_{h_i}$ such that $\partial_{h_i} (\pi^* g) = 0$, and $\partial_{h_i} (h_j) = \delta_{ij}$. 
It is easy to check that $\{ \partial_{h_i}, \wt{X}_i\}_{i=0}^{n}$ is a local frame on $T^*M$. We call such a frame the \emph{coordinate lifted frame}.

\begin{proof}[Proof of Proposition~\ref{p:riemcan} and Lemma~\ref{l:Riemanncurv}]
Point (i) is trivial and follows from the definition of the coordinate lifted frame $\partial_{h_i}$. In order to prove point (ii), we compute explicitly
\begin{equation}
\vec{H} = \sum_{i=1}^n \left(h_i \wt{X}_i + \sum_{j,k=1}^n h_i c_{ij}^k h_k \partial_{h_j}\right) = \sum_{i=1}^n\left( h_i \wt{X}_i + \sum_{j,k=1}^n h_i \Gamma_{ij}^k h_k \partial_{h_j}\right),
\end{equation}
where we used the identities $c_{ij}^k = \Gamma_{ij}^k - \Gamma_{ji}^k$ and $\Gamma_{ij}^k = -\Gamma_{ik}^j$. Then, a direct computation gives
\begin{equation}\label{eq:FiRiem}
F_i = - [\vec{H},\partial_{h_i}] = \wt{X}_i + \sum_{j,k=1}^n h_k \left(\Gamma_{ij}^k +\Gamma_{kj}^i \right) \partial_{h_j} = \wt{X}_i + \sum_{j,k=1}^n h_k \Gamma_{ij}^k \partial_{h_j},
\end{equation}
where we used the fact that, for a parallely transported frame, $\sum_{k=1}^n h_k \Gamma_{kj}^i = 0$ and we suppressed the explicit evaluation at $\lambda(t)$. Now we are ready to prove point (ii). Indeed $\sigma_{\lambda(t)}(\partial_{h_i},\partial_{h_j}) = 0$, since $\ve_\lambda$ is Lagrangian for all $\lambda$. Then
\begin{equation}
\sigma_{\lambda(t)}(\partial_{h_i},-[\vec{H},\partial_{h_j}]) = -\langle X_i|\pi_*[\vec{H},\partial_{h_j}]\rangle = \delta_{ij},
\end{equation}
where we used that $\pi_*[H,\partial_{h_j}] = - X_j$, and that for any vertical vector $\xi \in \ve_\lambda$ and $\eta \in T_{\lambda}(T^*M)$, $\sigma(\xi,\eta) = \langle \xi|\pi_*\eta\rangle$, where we identified $\xi$ with an element of $T_{\pi(\lambda)} M$ through the scalar product. Finally, by using the r.h.s. of Eq.~\eqref{eq:FiRiem}, we obtain
\begin{equation}
\sigma(F_i,F_j) = \sum_{k=1}^n \left(\Gamma_{ij}^k h_k - \Gamma_{ji}^k h_k - h_k c_{ij}^k \right)= \sum_{k=1}^n \langle h_k X_k|\nabla_{X_i} X_j - \nabla_{X_j} X_i - [X_i,X_j]\rangle =0,
\end{equation}
where we suppressed the explicit dependence on $t$ and the last equality is implied by the vanishing of the torsion of Levi-Civita connection. For what concerns point (iii), the first structural equation is the definition of $F_i$. By taking the derivative of $F_i$, we obtain
\begin{equation}
\dot{F}_i = [\vec{H}, F_i] = \sum_{\ell,k,j=1}^n h_\ell h_k \langle\nabla_{X_i} \nabla_{X_\ell} X_k-\nabla_{X_\ell} \nabla_{X_i} X_k - \nabla_{[X_i,X_\ell]} X_k|X_j\rangle E_j.
\end{equation}
In particular, this implies Lemma~\ref{l:Riemanncurv}, since
\begin{equation}
R_{ij}(t) = \sum_{\ell,j=1}^n h_\ell h_k \langle \nabla_{X_i}\nabla_{X_\ell} X_k - \nabla_{X_\ell} \nabla_{X_i} X_k - \nabla_{[X_i,X_\ell]}X_k|X_j\rangle = \langle R^\nabla(X_i,\dot{\gamma})\dot{\gamma}|X_j\rangle,
\end{equation}
by definition of Riemann tensor, and the fact that $\dot{\gamma}(t) = \sum_{i=1}^n h_i(\lambda(t)) X_i|_{\gamma(t)}$. Finally, let $\wt{E}_i,\wt{F}_j$ be any smooth moving frame along $\lambda(t)$ satisfying (i)-(iii). We can write, in full generality
\begin{equation}
\wt{E}_i = \sum_{j=1}^n A_{ij}(t)E_j + B_{ij}(t) F_j, \qquad \wt{F}_i = \sum_{j=1}^n C_{ij}(t)E_j + D_{ij}(t) F_j,
\end{equation}
for some smooth families of $n\times n$ matrices $A(t),B(t),C(t),D(t)$, where the frame is understood to be evaluated at $\lambda(t)$. By imposing conditions (i)-(iii), we obtain that the latter are actually constant, orthogonal matrices, and $B=C=0$, thus proving the uniqueness property.
\end{proof}

\paragraph*{Acknowledgements.}
The authors are grateful to Andrei Agrachev for fruitful discussions. The authors have been supported by Institut Henri Poincar\'e, Paris (RIP program), where most of this research has been carried out. The authors have been supported by the European Research Council, ERC StG 2009 ``GeCoMethods'', contract number 239748. The second author has been supported by INdAM (GDRE CONEDP). 

\bibliographystyle{abbrv}
\bibliography{biblio-comparison}

\end{document}